\documentclass[12pt, eqno, twoside]{article}
\usepackage{graphicx}
\usepackage{amscd}
\usepackage[matrix,arrow,curve]{xy}
\usepackage{amssymb,amsmath}
\usepackage{amsmath,amsfonts,amssymb,latexsym}
\usepackage{fancyhdr}
\newtheorem{theo}{Theorem}[section]
\newtheorem{mydf}{Definition}[section]
\newtheorem{prop}{Proposition}[section]
\newtheorem{cor}{Corollary}[section]
\newtheorem{lem}{Lemma}[section]

\newtheorem{scho}{Scholium}[section]

\newenvironment{proof}[1][Proof]{\noindent\textbf{#1.} }{\
\rule{0.5em}{0.5em}}

\begin{document}

\pagestyle{fancy}
\fancyhead{} 
\fancyhead[EC]{\small\it Anastasios Mallios, \ Patrice P.
Ntumba}%
\fancyhead[EL,OR]{\thepage} \fancyhead[OC]{\small\it Symplectic
Reduction of Sheaves of $\mathcal{A}$-modules}%
\fancyfoot{} 
\renewcommand\headrulewidth{0.5pt}
\addtolength{\headheight}{2pt} 

\title{\Large{Symplectic Reduction of Sheaves of $\mathcal{A}$-modules}}
\author{Anastasios Mallios, \ Patrice P.
Ntumba\footnote{Is the corresponding author for the paper}}

\date{}
\maketitle

\begin{abstract} Given an arbitrary sheaf $\mathcal{E}$ of $\mathcal{A}$-modules
(or $\mathcal{A}$-module in short) on a topological space $X$, we
define \textit{annihilator sheaves} of sub-$\mathcal{A}$-modules of
$\mathcal{E}$ in a way similar to the classical case, and obtain
thereafter the analog of the \textit{main theorem}, regarding
classical annihilators in module theory, see Curtis[\cite{curtis},
pp. 240-242]. The familiar classical properties, satisfied by
annihilator sheaves, allow us to set clearly the
\textit{sheaf-theoretic version} of \textit{symplectic reduction},
which is the main goal in this paper.
\end{abstract}

\maketitle

\section*{Introduction} This paper is part of our ongoing project
of \textit{algebraizing} classical symplectic geometry using the
tools of abstract differential geometry ($\grave{a}$ la Mallios).
Our main reference as far as abstract differential geometry is
concerned is the first author's book \cite{mallios}. For the sake
of convenience, we recall here some of the objects of abstract
differential geometry that recur all throughout.

Let $X$ be a topological space. A \textit{sheaf of
$\mathbb{C}$-algebras} or a \textit{$\mathbb{C}$-algebra sheaf},
on $X$, is a triple $\mathcal{A}\equiv (\mathcal{A}, \tau, X)$
satisfying the following conditions: \begin{enumerate} \item
[{$(i)$}] $\mathcal{A}$ is a sheaf of rings. \item [{$(ii)$}]
Fibers $\mathcal{A}_x\equiv \tau^{-1}(x)$, $x\in X$, are
$\mathbb{C}$-algebras. \item [{$(iii)$}] The \textit{scalar
multiplication} in $\mathcal{A}$, viz. the map \[\mathbb{C}\times
\mathcal{A}\longrightarrow \mathcal{A}: (c, a)\longmapsto c\cdot
a\in \mathcal{A}_x\subseteq \mathcal{A}\]with $\tau(a)= x\in X$,
is continuous; in this mapping, $\mathbb{C}$ is assumed to carry
the discrete topology.
\end{enumerate} The triple $(\mathcal{A}, \tau, X)$ is called a
\textit{unital} $\mathbb{C}$-algebra sheaf if the individual
fibers of $\mathcal{A}$, $\mathcal{A}_x$, $x\in X$, are unital
$\mathbb{C}$-algebras. A pair $(X, \mathcal{A})$, with
$\mathcal{A}$ assumed to be unital and commutative, is called a
\textit{$\mathcal{C}$-algebraized space}. Next, suppose that
$\mathcal{A}\equiv (\mathcal{A}, \tau, X)$ is a unital
$\mathbb{C}$-algebra sheaf on $X$. A \textit{sheaf of
$\mathcal{A}$-modules} (or an \textit{$\mathcal{A}$-module}), on
$X$, is a sheaf, $\mathcal{E}\equiv (\mathcal{E}, \rho, X)$, on
$X$ such that the following properties hold:
\begin{enumerate}\item [{$(iv)$}] $\mathcal{E}$ is a sheaf of
abelian groups on $X$. \item [{$(v)$}] Fibers $\mathcal{E}_x$,
$x\in X$, of $\mathcal{E}$ are $\mathcal{A}_x$-modules. \item
[{$(vi)$}] The \textit{left action} $\mathcal{A}\circ
\mathcal{E}\longrightarrow \mathcal{E}$, described by \[(a,
z)\longmapsto a\cdot z\in \mathcal{E}_x\subseteq
\mathcal{E},\]with $\tau(a)= \rho(z)= x\in X$, is continuous.
\end{enumerate}The sheaf-theoretic version of the classical notion
of a \textit{dual module} is defined in this manner: Given a
$\mathbb{C}$-algebraized space $(X, \mathcal{A})$ and an
$\mathcal{A}$-module $\mathcal{E}$ on $X$, the
$\mathcal{A}$-module (on $X$) \[\mathcal{E}^\ast:=
\mathcal{H}om_\mathcal{A}(\mathcal{E}, \mathcal{A})\]is called the
\textit{dual $\mathcal{A}$-module} of $\mathcal{E}$. For tow given
$\mathcal{A}$-modules on a topological space $X$,
$\mathcal{H}om_\mathcal{A}(\mathcal{E}, \mathcal{F})$ is the
$\mathcal{A}$-module generated on $X$ by the (complete) presheaf,
given by $U\longmapsto Hom_{\mathcal{A}|_U}(\mathcal{E}|_U,
\mathcal{F}|_U)$, where $U$ runs over the open subsets of $X$; the
restriction maps of this presheaf are quite obvious. A most
familiar consequence regarding dual $\mathcal{A}$-modules is that
given a free $\mathcal{A}$-module $\mathcal{E}$ of finite rank on
$X$, one has $\mathcal{E}= \mathcal{E}^\ast$, within an
$\mathcal{A}$-isomorphism.

Section $1$ is concerned with \textit{annihilator sheaves} of
sub-$\mathcal{A}$-modules of arbitrary $\mathcal{A}$-modules on
one hand, and $\varphi$-annihilator sheaves, i.e. annihilator
sheaves (of sub-$\mathcal{A}$-modules) with respect to a
non-degenerate bilinear $\mathcal{A}$-morphism $\varphi:
\mathcal{E}\oplus \mathcal{F}\longrightarrow \mathcal{A}$, where
$\mathcal{E}$ and $\mathcal{F}$ are free $\mathcal{A}$-modules of
finite rank. The sheaf-theoretic version of the main theorem on
classical annihilators is examined. The section ends with the
interesting result that given a sub-$\mathcal{A}$-module
$\mathcal{F}$ of an $\mathcal{A}$-module $\mathcal{E}$, the dual
$\mathcal{A}$-module $(\mathcal{E}/\mathcal{F})^\ast$ is
$\mathcal{A}$-isomorphic to the annihilator $\mathcal{F}^\perp$ of
$\mathcal{F}$.

Section $2$ deals with properties of exterior \textit{rankwise}
$\mathcal{A}$-$2$-forms. We provide another proof for the
\textit{affine Darboux theorem}. The proof is derived from E.
Cartan\cite{cartan}.

Section $3$, which is the last section, outlines the
\textit{symplectic reduction} of an $\mathcal{A}$-module
$\mathcal{E}$ by a co-isotropic sub-$\mathcal{A}$-module
$\mathcal{F}$ of $\mathcal{E}$; the $\mathcal{A}$-module
$\mathcal{E}$ carries a symplectic $(\mathcal{A}-)$ structure,
given by the $\mathcal{A}$-morphism $\omega: \mathcal{E}\oplus
\mathcal{E}\longrightarrow \mathcal{A}$.

\section{Annihilator Sheaves}
\begin{mydf} \emph{Let $(\mathcal{S}, \pi, X)$ be a sheaf. By a \textbf{subsheaf}
of $\mathcal{S}$, we mean a sheaf $\mathcal{E}$ on $X$,
\textit{generated by a presheaf} $(E(U), \sigma^U_V)$ which is such
that, for all open $U\subseteq X$ and open $V\subseteq U$,
\begin{itemize}
\item $E(U)\subseteq
\mathcal{S}(U)$,
\item $\sigma^U_V= \rho^U_V|_{E(U)}$,
\end{itemize}
where $(\mathcal{S}(U)\equiv \Gamma(U, \mathcal{S}),
\rho^U_V)\equiv \Gamma(\mathcal{S})$ is the $($complete$)$
presheaf of sections of the sheaf $\mathcal{S}$, cf.
Mallios[\cite{mallios}, Lemma 11.1, p. 48].}\hfill$\square$
\label{subs}\end{mydf}

We can now define the notion of sub-$\mathcal{A}$-module of a given
$\mathcal{A}$-module, which will be of use in the sequel.

\begin{mydf}
\emph{A subsheaf $\mathcal{E}$ of an $\mathcal{A}$-module
$\mathcal{S}$, defined on a topological space $X$, is called a
\textbf{sub-$\mathcal{A}$-module} of $\mathcal{S}$ if $\mathcal{E}$
is an \textit{$\mathcal{A}$-module} and the \textit{inclusion $i:
\mathcal{E}\subseteq \mathcal{S}$ is an $\mathcal{A}$-morphism}.}
\hfill$\square$
\end{mydf}

\begin{lem} Subsheaves are open subsets, and conversely.\label{eqd}\end{lem}

\begin{proof}
Let $\mathcal{S}$ be a sheaf on $(X, \mathcal{T})$, $\mathcal{E}$
a subsheaf, of $\mathcal{S}$, generated by the presheaf $(E(U),
\sigma^U_V)$, and let us denote by $\Re$ the set $\bigcup\{E(U):\
U\in \mathcal{T}\}$. According to Mallios[\cite{mallios}, Theorem
3.1, p.14], the family
\[\mathcal{B}=\{ s(U):\ s\in \Re\
\mbox{and $U\in \mathcal{T}$ with $U= \mbox{Dom}(s)$}\}\]is a
basis for the topology of $\mathcal{E}$, with respect to which
$\mathcal{E}$ is a sheaf on $X$. But $E(U)\subseteq
\mathcal{S}(U)$ for every open $U\subseteq X$, therefore, for all
$s\in \Re$, $s(U)$ is open in $\mathcal{S}$, and thus
$\bigcup\mathcal{B}= \mathcal{E}$ is open in $\mathcal{S}$, as
desired.

For the converse, see Mallios[\cite{mallios}, p. 5].
\end{proof}

It follows from Lemma \ref{eqd} that Definition \ref{subs} and
Mallios' definition of subsheaf, see Mallios[\cite{mallios}, p. 5],
are equivalent.

\begin{mydf}\emph{ Let $\mathcal{E}$ be an $\mathcal{A}$-module on a
topological space $X$, and $\mathcal{F}$ a sub-$\mathcal{A}$-module
of $\mathcal{E}$. Assume that $(\mathcal{E}(U), \sigma^U_V)$ is the
(complete) presheaf of sections of $\mathcal{E}$. By the
\textbf{$\mathcal{A}$-annihilator sheaf} (or \textbf{sheaf of
$\mathcal{A}$-annihilators}, or just
\textbf{$\mathcal{A}$-annihilator}) of $\mathcal{F}$, we mean the
sheaf generated by the \textit{presheaf}, given by the
correspondence
\[U\longmapsto \mathcal{F}(U)^\perp,\]where $U$ is an open subset
of $X$ and
\[\mathcal{F}(U)^\perp= \{t\in \mathcal{E}^\ast(U):\ t(s)=0\
\mbox{for all $s\in \mathcal{E}(U)$}\},\]along with restriction
maps
\[(\rho^\perp)^U_V: \mathcal{F}(U)^\perp\longrightarrow
\mathcal{F}(V)^\perp\]such that
\[(\rho^\perp)^U_V:=
(\sigma^\ast)^U_V|_{\mathcal{F}(U)^\perp},\]with the
$(\sigma^\ast)^U_V: \mathcal{E}^\ast(U)\longrightarrow
\mathcal{E}^\ast(V)$ being the restriction maps for the \textit{dual
presheaf} $(\mathcal{E}^\ast(U), (\sigma^\ast)^U_V)$. We denote by
\[\mathcal{F}^\perp\]the annihilator sheaf of
$\mathcal{F}$.}\hfill$\square$ \label{sub}
\end{mydf}

It follows from Definition \ref{sub}, that the annihilator
$\mathcal{F}^\perp$ of a sub-$\mathcal{A}$-module $\mathcal{F}$ of
an $\mathcal{A}$-module $\mathcal{E}$ is a subsheaf of the dual
$\mathcal{A}$-module $\mathcal{E}^\ast$.

\begin{lem} Let $\mathcal{E}$ be an $\mathcal{A}$-module on a
topological space $X$, and $\mathcal{F}$ a sub-$\mathcal{A}$-module
of $\mathcal{E}$. Then, the correspondence
\[U\longmapsto \mathcal{F}(U)^\perp\]along with maps
$(\rho^\perp)^U_V$, as defined above, yields a \textsf{complete
presheaf of $\mathcal{A}$-modules on $X$}. \label{lsub}\end{lem}

\begin{proof}
First, we notice that it is immediate that for every open
$U\subseteq X$, $\mathcal{F}(U)^\perp$ is an
$\mathcal{A}(U)$-module and that $(\mathcal{F}(U)^\perp,
(\rho^\perp)^U_V)$ is a presheaf of $\mathcal{A}$-modules on $X$.
To see that the presheaf $(\mathcal{F}(U)^\perp,
(\rho^\perp)^U_V)$ is complete, let us fix an open subset $U$ of
$X$ and an open covering $\mathcal{U}= \{U_\alpha\}_{\alpha\in I}$
of $U$. Next, let $s$, $t$ be two elements of
$\mathcal{F}(U)^\perp$ such that
\[(\rho^\perp)^U_{U_\alpha}(s)\equiv s_\alpha= t_\alpha \equiv
(\rho^\perp)^U_{U_\alpha}(t),\]for every $\alpha\in I$. Since
$\mathcal{F}(U)^\perp\subseteq \mathcal{E}^\ast(U)$, so $s, t\in
\mathcal{E}^\ast(U)$; $(\mathcal{E}^\ast(U), (\rho^\ast)^U_V)$
being the complete presheaf of sections associated with the sheaf
$\mathcal{E}^\ast$ (for the completeness of $(\mathcal{E}^\ast(U),
(\rho^\ast)^U_V)$, see Mallios[\cite{mallios}, (6.4), Definition
6.1, p.134, and (5.1), p.298]), and
\[\begin{array} {lll} (\rho^\perp)^U_{U_\alpha}(s)=
(\rho^\ast)^U_{U_\alpha}(s) & \mbox{and} &
(\rho^\perp)^U_{U_\alpha}(t)=
(\rho^\ast)^U_{U_\alpha}(t)\end{array},\]we have that $s=t$.
Therefore, axiom $(S1)$ (cf. Mallios[\cite{mallios}, p.46]) is
satisfied\footnote{We can however notice here that the presheaf of
Lemma \ref{lsub} can also be considered as a \textit{functional}
presheaf by its very definition, hence axiom $(S1)$, see
Mallios[\cite{localizing}relation 1.12, p. 81]}.

Now, let us check that axiom $(S2)$, see Mallios[\cite{mallios},
p.47], is also satisfied. So, let $U$ and $\mathcal{U}=
\{U_\alpha\}_{\alpha\in I}$ be as above. Furthermore, let
$(t_\alpha)\in \prod_\alpha\mathcal{F}(U_\alpha)^\perp$ be such that
for any $U_{\alpha\beta}\equiv U_\alpha\cap U_\beta\neq \emptyset$
in $\mathcal{U}$, one has
\[(\rho^\perp)^{U_\alpha}_{U_{\alpha\beta}}(t_\alpha)\equiv
t_\alpha|_{U_{\alpha\beta}}= t_\beta|_{U_{\alpha\beta}}\equiv
(\rho^\perp)^{U_\beta}_{U_{\alpha\beta}}(t_\beta).\]Since
$(\rho^\perp)^{U_\alpha}_{U_{\alpha\beta}}=
(\rho^\ast)^{U_\alpha}_{U_{\alpha\beta}}|_{\mathcal{F}(U_\alpha)^\perp}$,
$\mathcal{F}(U_\alpha)^\perp\subseteq \mathcal{E}^\ast(U_\alpha)$
for all $\alpha, \beta\in I$, and the presheaf
$(\mathcal{E}^\ast(U), (\rho^\ast)^U_V)$ is \textit{complete}, there
exists an element $t\in \mathcal{E}^\ast(U)$ such that
\[(\rho^\ast)^U_{U_\alpha}(t)\equiv t|_{U_\alpha}= t_\alpha,\]for
every $\alpha\in I$. We should now show that $t$ is indeed an
element of $\mathcal{F}(U)^\perp$. To this end, suppose that there
exists an $s\in \mathcal{E}(U)$ such that $t(s)\neq 0\in
\mathcal{A}(U)$; this implies that for some $\alpha\in I$,
\[(\rho^\ast)^U_{U_\alpha}(t)(\rho^U_{U\alpha}(s))\equiv
t|_{U_\alpha}(s|_{U_\alpha})= t_\alpha(s_\alpha)\neq 0,\]which is
impossible as $t_\alpha\in \mathcal{F}(U_\alpha)^\perp$ and
$s_\alpha\in \mathcal{E}(U_\alpha)$. Thus, $t\in
\mathcal{F}(U)^\perp$; hence $(S2)$ is satisfied.
\end{proof}

By virtue of Proposition 11.1, see Mallios[\cite{mallios}, p.51],
if $\mathcal{F}$ is a sub-$\mathcal{A}$-module of an
$\mathcal{A}$-module $\mathcal{E}$, then
\begin{equation}\mathcal{F}^\perp(U)= \mathcal{F}(U)^\perp\label{anni}\end{equation} within an
$\mathcal{A}(U)$-isomorphism.

From the relation (\ref{anni}), we have the following corollary.

\begin{cor}Let $\mathcal{E}$ be an $\mathcal{A}$-module on a
topological space, and $\mathcal{F}$ a sub-$\mathcal{A}$-module of
$\mathcal{E}$. Then \[\mathcal{F}^\perp= \{z\in \mathcal{E}^\ast:\
z(u)=0 \ \mbox{for all $u\in \mathcal{F}$}\}.\]
\end{cor}

\begin{lem} Let $\mathcal{E}$ be an $\mathcal{A}$-module on a
topological space $X$, and $U$ an open subset of $X$. Then,
\[\mathcal{E}^\ast|_U= (\mathcal{E}|_U)^\ast\]within an
$\mathcal{A}|_U$-isomorphism of the sheaves in question.
\label{lem02}
\end{lem}

\begin{proof}
For any open subset $V\subseteq U$, we have
\begin{eqnarray*}(\mathcal{E}^\ast|_U)(V)\equiv
(\mathcal{H}om_\mathcal{A}(\mathcal{E}, \mathcal{A})|_U)(V)=
\mathcal{H}om_\mathcal{A}(\mathcal{E}, \mathcal{A})(V) =
\mbox{Hom}_{\mathcal{A}|_V}(\mathcal{E}|_V, \mathcal{A}|_V),
\end{eqnarray*}and
\[(\mathcal{E}|_U)^\ast(V)\equiv
\mathcal{H}om_{\mathcal{A}|_U}(\mathcal{E}|_U, \mathcal{A}|_U)(V)=
\mbox{Hom}_{\mathcal{A}|_V}(\mathcal{E}|_V, \mathcal{A}|_V).\]We
thus conclude that since these two sheaves have isomorphic (local)
sections, they are $\mathcal{A}|_U$-isomorphic.
\end{proof}

Using the language of Category Theory, see MacLane\cite{maclane} for
Category Theory, Lemma \ref{lem02} infers that the
\textit{dual-$\mathcal{A}$-module functor} (cf.
Mallios[\cite{mallios}, (5.20), p.301]) commutes with the
\textit{restriction-(over $U$)-of-$\mathcal{A}$-modules functor}.
Schematically, we have the commutative diagram
\[\xymatrix{\mathcal{E}\ar[r]\ar[d] &
\mathcal{E}^\ast\ar[d]\\ \mathcal{E}|_U \ar[r] &
\mathcal{E}^\ast|_U= (\mathcal{E}|_U)^\ast.}\]

The following definition hinges on Lemma \ref{lem02} and
Mallios[\cite{mallios}, (5.2), p.298].

\begin{mydf}
\emph{Let $\mathcal{E}$ and $\mathcal{F}$ be $\mathcal{A}$-modules
on a topological space $X$, and let $\varphi\in
\mbox{Hom}_{\mathcal{A}|_U}(\mathcal{E}|_U, \mathcal{F}|_U)=
\mathcal{H}om_\mathcal{A}(\mathcal{E}, \mathcal{F})(U)$ with $U$ an
open subset of $X$. The \textbf{adjoint} of $\varphi$ is the sheaf
$\mathcal{A}|_U$-morphism
\begin{eqnarray*}\lefteqn{\varphi^\ast\equiv (\varphi^\ast_V)_{U\supseteq V, open}\equiv ((\varphi_V)^\ast)_{U\supseteq V,
open}\in \mbox{Hom}_{(\mathcal{A}|_U)^\ast}((\mathcal{F}|_U)^\ast,
(\mathcal{E}|_U)^\ast)}\\ & & =
\mbox{Hom}_{\mathcal{A}^\ast|_U}(\mathcal{F}^\ast|_U,
\mathcal{E}^\ast|_U)=
\mbox{Hom}_{\mathcal{A}|_U}(\mathcal{F}^\ast|_U,
\mathcal{E}^\ast|_U)= \mathcal{H}om_\mathcal{A}(\mathcal{F}^\ast,
\mathcal{E}^\ast)(U)
\end{eqnarray*}
such that for all $\omega\equiv (\omega_W)_{V\supseteq W, open}\in
\mbox{Hom}_{\mathcal{A}|_V}(\mathcal{F}|_V, \mathcal{A}|_V)=
\mathcal{F}(V)$, one has \[\varphi^\ast_V(\omega)= (\omega\circ
\varphi)|_{\mathcal{E}|_V},\]that is \[(\varphi^\ast_V(\omega))_W=
\omega_W\circ \varphi_W\]for all open $W\subseteq
V$.}\hfill$\square$\label{def04}
\end{mydf}

\begin{scho}
\emph{If the $\mathcal{A}$-modules $\mathcal{E}$ and $\mathcal{F}$ in
Definition \ref{def04} are \textit{vector sheaves} on $X$, then
concerning the \textit{adjoint} $\varphi^\ast$ of an
$\mathcal{A}$-morphism $\varphi\in Hom_\mathcal{A}(\mathcal{E},
\mathcal{F})$, we have
\begin{equation}\label{adjoint}\varphi^\ast\in
\mathcal{H}om_\mathcal{A}(\mathcal{F}^\ast, \mathcal{E}^\ast)=
\mathcal{H}om_\mathcal{A}(\mathcal{E}, \mathcal{F})^\ast=
\mathcal{H}om_\mathcal{A}(\mathcal{F},
\mathcal{E}),\end{equation}where as usual the displayed equalities
of (\ref{adjoint}) are $\mathcal{A}$-isomorphisms of the
$\mathcal{A}$-modules involved. For these foregone
$\mathcal{A}$-isomorphisms, see Mallios[\cite{mallios}, Corollary
6.3, p.306].}
\end{scho}

\begin{theo} Let $\mathcal{E}$ and $\mathcal{F}$ be an
$\mathcal{A}$-module on a topological space $X$, and $\varphi\in
\mbox{Hom}_{\mathcal{A}|_U}(\mathcal{E}|_U, \mathcal{F}|_U)$, where
$U$ is some open subset of $X$. Then,
\[(\mbox{im}\varphi)^\perp= \ker \varphi^\ast\]within an
$\mathcal{A}|_U$-isomorphism, that is for every open subset
$V\subseteq U$, we have
\[(\mbox{im}\varphi_V)^\perp= \ker \varphi_V^\ast\]within an
$\mathcal{A}(V)$-isomorphismof modules.
\end{theo}

\begin{proof}
Let $\omega\in (\mathcal{F}|_U)^\ast= \mathcal{F}^\ast|_U$. Then,
\[
\begin{array}{lll} \omega\in \ker \varphi^\ast_V & \Leftrightarrow &
\varphi^\ast_V\omega= 0\\ & \Leftrightarrow & \omega_W\circ
\varphi_W=0 \ \mbox{for all open $W\subseteq V$}\\ & \Leftrightarrow
& (\omega_W\circ \varphi_W)(s)=0 \ \mbox{for all open $W\subseteq V$
and $s\in \mathcal{E}(W)$}\\ & \Leftrightarrow & \omega_W(t)=0\
\mbox{for all open $W\subseteq V$ and $t\in
\varphi_W(\mathcal{E}(W))$}\\ & \Leftrightarrow & \omega_W\in
\mbox{im} (\mathcal{E}(W))^\perp\ \mbox{for all open $W\subseteq V$}\\
& \Leftrightarrow & \omega\in (\mbox{im} \mathcal{E}(V))^\perp=
(\mbox{im} \varphi_V)^\perp
\end{array}
\]Thus, we have the sought $\mathcal{A}(V)$-isomorphism for every
open $V\subseteq U$; the proof is thus finished.
\end{proof}

The notion of an annihilator sheaf may be generalized by
considering any two $\mathcal{A}$-modules that are \textit{``dual"
with respect to some non-degenerate bilinear $\mathcal{A}$-form};
i.e. $\mathcal{A}$-isomorphic within an $\mathcal{A}$-isomorphism
determined by the non-degenerate bilinear $\mathcal{A}$-form
considered. Before we define the notion of non-degenerate bilinear
$\mathcal{A}$-morphism, we would like to make the following
convention: In fact, let $\mathcal{E}$ and $\mathcal{F}$ be
$\mathcal{A}$-modules on a topological space $X$. An
$\mathcal{A}$-morphism $\varphi\in
\mbox{Hom}_\mathcal{A}(\mathcal{E}, \mathcal{F})$ will be denoted
$\varphi\equiv (\varphi^U)_{U\in \mathcal{T}}$ or $\varphi\equiv
(\varphi_U)_{U\in \mathcal{T}}$. These notations will depend on
the situation at hand, but this will be done for the sole purpose
of making indices a lot easier to handle.

\begin{mydf} \emph{Let $\mathcal{E}$ and $\mathcal{F}$ be
$\mathcal{A}$-modules on $X$. A bilinear $\mathcal{A}$-morphism
$\varphi\equiv (\varphi^U)_{X\supseteq U,\ open}: \mathcal{E}\oplus
\mathcal{F}\longrightarrow \mathcal{A}$ is said to be
\textbf{non-degenerate} if for every open subset $U$ of $X$, the
following conditions hold.
\[\begin{array}{ll} \varphi^U(s, t)= 0 & \mbox{for all $t\in
\mathcal{F}(U)$ implies that $s=0$}\end{array}\]and
\[\begin{array}{ll} \varphi^U(s, t)= 0 & \mbox{for all $s\in
\mathcal{E}(U)$ implies that $t=0$.}\end{array}\]}\hfill$\square$
\label{def05}\end{mydf}

\begin{prop} Let $\mathcal{E}$ and $\mathcal{F}$ be each a free
$\mathcal{A}$-module of finite rank, and $\varphi:
\mathcal{E}\oplus \mathcal{F}\longrightarrow \mathcal{A}$ a
non-degenerate bilinear $\mathcal{A}$-morphism. For every
$($local$)$ section $t\in \mathcal{F}(U)$, let $\alpha^U_t:
\mathcal{E}(U)\longrightarrow \mathcal{A}(U)$ be the mapping
\begin{equation}\alpha^U_t(s):= \varphi^U(s, t)\label{eq1}\end{equation}for all $s\in
\mathcal{E}(U)$. Then, the mapping $\alpha^U_t$, as defined above,
is an element of $\mathcal{E}(U)^\ast= \mathcal{E}^\ast(U)=
\mathcal{E}(U)$, where the previous equalities are derived from
Mallios$[$$\cite{mallios}$, $($3.14$)$, p.122, $($5.2.1$)$,
p.298$]$. On the other hand, the mapping
\[\vartheta: \mathcal{F}\longrightarrow \mathcal{E}\]defined
by
\[\vartheta^U(t)= \alpha^U_t\]yields an $\mathcal{A}$-isomorphism of
$\mathcal{F}$ onto $\mathcal{E}$. \label{prop01}
\end{prop}

\begin{proof} That $\alpha^U_t$, as given in relation (\ref{eq1}), is an
element of $\mathcal{E}^\ast(U)$ is clear. It is also easy to see
that the mapping
\[\vartheta^U(t)= \alpha^U_t\]is a module morphism of
$\mathcal{F}(U)$ into $\mathcal{E}^\ast(U)$, where, as above, $U$
is an open subset of $X$. Now, let us show that every component
$\vartheta^U$ of $\vartheta$ is an isomorphism of
$\mathcal{A}(U)$-modules $\mathcal{F}(U)$ and $\mathcal{E}(U)$. To
this purpose, suppose first that $\vartheta^U(t)=
\vartheta^{U}(t')$, for $t, t'\in \mathcal{F}(U)$; then
\[\varphi^U(s, t)= \varphi^U(s, t')\]for all $s\in
\mathcal{E}(U)$. By the bilinearity and non-degeneracy of
$\varphi$, we have that $t=t'$. Therefore, $\vartheta^U$ is
one-to-one.

Now, suppose that the ranks of the free $\mathcal{A}$-modules
$\mathcal{E}$ and $\mathcal{F}$ are $n$ and $m$, respectively,
that is $\mathcal{E}= \mathcal{A}^n$ and $\mathcal{F}=
\mathcal{A}^m$ within $\mathcal{A}$-isomorphisms. Since
$\vartheta^U: \mathcal{F}(U)= \mathcal{A}^m(U)\longrightarrow
\mathcal{A}^n(U)= \mathcal{E}(U)$ is one-to-one and
$\mathcal{A}^k(U)$ is a free module for every $k\in \mathbb{N}$,
it follows that \[\dim \mathcal{E}(U)\geq \dim \mathcal{F}(U).\]By
a similar argument, one shows that there exists a one-to-one
$\mathcal{A}(U)$-morphism of $\mathcal{E}(U)$ into
$\mathcal{F}(U)$, and therefore \[\dim \mathcal{F}(U)\geq \dim
\mathcal{E}(U);\]hence \[\dim \mathcal{E}(U)= \dim
\mathcal{F}(U);\]which implies that $\vartheta^U$ is onto.
Therefore, for every open $U\subseteq X$, $\mathcal{E}(U)$ is
$\mathcal{A}(U)$-isomorphic to $\mathcal{F}(U)$. The restriction
maps of the associated complete presheaves of sections $(\Gamma(U,
\mathcal{E}), \rho^U_V)$ and $(\Gamma(U, \mathcal{F},
\lambda^U_V)$ can be chosen in such a way that the diagram
\[\xymatrix{\mathcal{F}(U)\ar[r]^{\vartheta^U}\ar[d]_{\lambda^U_V} &
\mathcal{E}(U)\ar[d]^{\rho^U_V}\\ \mathcal{F}(V)
\ar[r]_{\vartheta^V} & \mathcal{E}(V)}\]commutes. Hence,
$\vartheta\equiv (\vartheta^U)$ is an $\mathcal{A}$-isomorphism
between $\mathcal{F}$ and $\mathcal{E}$.
\end{proof}

From Proposition \ref{prop01}, we bring about the following
definition.

\begin{mydf} \emph{A pair of free $\mathcal{A}$-modules  $\mathcal{E}$ and
$\mathcal{F}$ are said to be \textbf{$\mathcal{A}$-dual} or just
\textbf{dual} with respect to a bilinear $\mathcal{A}$-morphism
$\varphi: \mathcal{E}\oplus \mathcal{F}\longrightarrow \mathcal{A}$
if $\varphi$ is non-degenerate. If we want to stress the fact that
$\varphi$ is the bilinear map with respect to which the free
$\mathcal{A}$-modules $\mathcal{E}$ and $\mathcal{F}$ are dual, we
shall say that $\mathcal{E}$ and $\mathcal{F}$ are
\textbf{$\varphi$-dual}.} \hfill$\square$
\end{mydf}

We also need the following definition.

\begin{mydf} \emph{Let $\mathcal{E}$ and $\mathcal{F}$ be \textit{$\varphi$-dual free
$\mathcal{A}$-modules}, $S\equiv (S^U)\in
\mbox{End}_\mathcal{A}\mathcal{E}:= (\mathcal{E}nd_\mathcal{A}
\mathcal{E})(X)$, and $T\equiv (T^U)\in
\mbox{End}_\mathcal{A}\mathcal{F}:=
(\mathcal{E}nd_\mathcal{A}\mathcal{F})(X)$. Then, $S$ and $T$ are
said to be \textbf{transposes} of each other provided that
\[\varphi^U(s, T^U(t))= \varphi^U(S^U(s), t)\]for all $s\in
\mathcal{E}(U)$ and $t\in \mathcal{F}(U)$.}\hfill$\square$
\end{mydf}

\begin{theo} Let $\mathcal{E}$ and $\mathcal{F}$ be free
$\mathcal{A}$-modules which are dual with respect to a bilinear
$\mathcal{A}$-morphism $\varphi: \mathcal{E}\oplus
\mathcal{F}\longrightarrow \mathcal{A}$. Moreover, let $S\equiv
(S^U)\in \mbox{End}_\mathcal{A}\mathcal{E}$; then there exists a
uniquely determined family \[T\equiv (T^U)\in \prod_{X\supseteq U,
open}\mbox{End}_{\mathcal{A}(U)}\mathcal{F}(U)\] such that for all
open $U\subseteq X$, $S^U$ and $T^U$ are transposes of each other.
Furthermore, if $T\equiv (T^U)$ is an $\mathcal{A}$-endomorphism
$\mathcal{F}\longrightarrow \mathcal{F}$, then $S$ and $T$ are
transposes of each other as sheaf morphisms.
\end{theo}

\begin{proof} We have to show that for all open $U\subseteq X$ and $t\in
\mathcal{F}(U)$, there exists a unique element $t'\in
\mathcal{F}(U)$ such that
\begin{equation}\begin{array}{ll}\varphi^U(s, t')= \varphi^U(S^U(s),
t), & s\in
\mathcal{E}(U),\end{array}\label{transpose}\end{equation}so that
we can define $T^U(t)$ to be $t'$. For each open $U\subseteq X$,
$S^U\in \mbox{End}_{\mathcal{A}(U)}\mathcal{E}(U)$; because of the
latter the mapping \[\psi^U: s\longmapsto \varphi^U(S^U(s), t)\]is
an element of $\mathcal{E}^\ast(U)= \mathcal{E}(U)$. By
Proposition \ref{prop01}, there exists a unique element $t'\in
\mathcal{F}(U)$ such that \[\vartheta^U(t')(=\alpha^U_{t'})=
\psi^U;\]so for all $s\in \mathcal{E}(U)$, we have \[\varphi^U(s,
t')= \varphi^U(S^U(s), t).\] Now, define $T^U:
\mathcal{F}(U)\longrightarrow \mathcal{F}(U)$ as $T^U(r)= r'$,
where $r'$ is the solution of the equation obtained  by
substituting $r$ for $t$ in (\ref{transpose}). Then, we have
\[\varphi^U(s, T^U(t))= \varphi^U(S^U(s), t)\]for all $s\in
\mathcal{E}(U)$ and $t\in \mathcal{F}(U)$. The mapping $T^U$ is an
$\mathcal{A}(U)$-endomorphism of $\mathcal{F}(U)$. (The details of
checking that $T^U\in \mbox{End}_{\mathcal{A}(U)}\mathcal{E}(U)$
are presented in the proof of Theorem (27.7), \cite{curtis},
p.239.)
\end{proof}

\begin{lem}
Let $\mathcal{E}$ and $\mathcal{F}$ be dual free
$\mathcal{A}$-modules, with respect to a non-degenerate bilinear
$\mathcal{A}$-morphism $\varphi$, and $\mathcal{G}$ a
sub-$\mathcal{A}$-module of $\mathcal{E}$. For all open
$U\subseteq X$, let
\[\mathcal{G}(U)^\perp:= \{t\in \mathcal{F}(U):\
\varphi^U(s, t)= 0,\ \mbox{for all $s\in \mathcal{G}(U)$}\}.\]
Moreover, let $(\sigma^\perp)^U_V:
\mathcal{G}(U)^\perp\longrightarrow \mathcal{G}(V)^\perp$, where
$V\subseteq U$, with $V$ and $U$ open in $X$, be mappings
\[(\sigma^\perp)^U_V:=
\rho^U_V|_{\mathcal{G}(U)^\perp},\]where the $\rho^U_V$ are the
restriction maps for the $($complete$)$ presheaf of sections
$(\mathcal{F}(U), \rho^U_V)$. Then, the correspondence
\begin{equation}U\longmapsto
\mathcal{G}(U)^\perp,\label{an}\end{equation} along with the maps
$(\sigma^\perp)^U_V$, yields a complete presheaf of
$\mathcal{A}$-modules on $X$.\label{lem01}
\end{lem}

\begin{proof} We first notice, by virtue of Proposition \ref{prop01}, that,
for all open subset $U\subseteq X$, $\mathcal{F}(U)$ is
$\mathcal{A}(U)$-isomorphic to $\mathcal{E}(U)$, so that
restricting a map such as $\rho^U_V: \mathcal{E}(U)\longrightarrow
\mathcal{E}(V)$ to $\mathcal{G}(U)^\perp$ makes sense.

The rest of the proof is similar to the proof of Lemma \ref{lsub}.
\end{proof}

Lemma \ref{lem01} makes the following definition rather natural.

\begin{mydf}
\emph{Let $\mathcal{E}$, $\mathcal{G}$ and $\mathcal{F}$ be as in
Lemma \ref{lem01}. We denote by
\[\mathcal{G}^\perp\]the sheaf on $X$ generated by the
$($complete$)$ presheaf defined by $(\ref{an})$. We call it the
\textbf{$\mathcal{A}$-annihilator sheaf} of $\mathcal{G}$ with
respect to the non-degenerate bilinear $\mathcal{A}$-morphism
$\varphi$ (or $\mathcal{A}_\varphi$-\textbf{annihilator}) of
$\mathcal{G}$.}\hfill$\square$
\end{mydf}

\begin{cor}
For any $\varphi$-dual free $\mathcal{A}$-modules $\mathcal{E}$
and $\mathcal{F}$ on $X$, the annihilator sheaf
$\mathcal{G}^\perp$ of a subsheaf $\mathcal{G}$ of $\mathcal{E}$,
as defined above, is an $\mathcal{A}$-module on $X$.
\end{cor}

We are now set for the \textit{main theorem} on
$\mathcal{A}$-annihilator sheaves. The results of the analog
theorem in classical module theory can be found in
Curtis[\cite{curtis}, pp. 240-242], Adkins and
Weintraub[\cite{adkins}, pp. 345-349]. But before we state the
theorem, we open a breach for the analog of a submodule of a
module $M$, generated by the set $\cup_{i\in I}M_i$, where every
$M_i$ is a submodule of $M$.

\begin{lem}
Let $\mathcal{E}$ be an $\mathcal{A}$-module on $X$, and
$(\mathcal{F}_i)_{i\in I}$ a family of sub-$\mathcal{A}$-modules of
$\mathcal{E}$. For every open $U\subseteq X$, let \[F(U):=
\langle\cup_{i\in I}\mathcal{F}_i(U)\rangle,\]that is $F(U)$ is the
$\mathcal{A}(U)$-submodule of $\mathcal{E}(U)$, generated by
$\cup_{i\in I}\mathcal{F}_i(U)$, i.e. $F(U)$ is the sum of the
family $(\mathcal{F}_i(U))_{i\in I}$. The presheaf, given by
\begin{equation}U\mapsto F(U):= \langle\cup_{i\in
I}\mathcal{F}_i(U)\rangle,\label{eq12}\end{equation}where $U$ runs
over the open subsets of $X$, along with restriction maps
$\sigma^U_V= {\rho^U_V}|_{\mathcal{F}(U)}$ $(( \mathcal{E}(U),
\rho^U_V)$ is the presheaf of sections of $\mathcal{E}$ $)$, is
complete. \label{lem03}
\end{lem}

\begin{proof}
That $(F(U), \sigma^U_V)$ is a presheaf of $\mathcal{A}$-modules on
$X$ is easy to see. To see that the presheaf $(F(U), \sigma^U_V)$ is
complete, we need check axioms $(S1)$ and $(S2)$ in
Mallios[\cite{mallios}, pp. 46-47]. It is easy to see that axiom
$(S1)$ is satisfied. To verify that axiom $(S2)$ is satisfied, let
$U$ be an open subset of $X$ and $\mathcal{U}= (U_\alpha)_{\alpha\in
I}$ an open covering of $\mathcal{U}$. Furthermore, let $(t_\alpha)=
\prod_\alpha F(U_\alpha)$ be such that for any
$U_{\alpha\beta}\equiv U_\alpha\cap U_\beta\neq \emptyset$ in
$\mathcal{U}$, one has
\[\sigma^{U_\alpha}_{U_{\alpha\beta}}(t_\alpha)\equiv
{t_\alpha}|_{U_{\alpha\beta}}= {t_\beta}|_{U_{\alpha\beta}}\equiv
\sigma^{U_\beta}_{U_{\alpha\beta}}(t_\beta) .\]Since
$\sigma^{U_\alpha}_{U_{\alpha\beta}}=
\rho^{U_\alpha}_{U_{\alpha\beta}}|_{F(U_{\alpha})}$,
$F(U_\alpha)\subseteq \mathcal{E}(U_\alpha)$ for all $\alpha,
\beta\in I$, and the presheaf $(\mathcal{E}(U), \rho^U_V)$ is
complete, there exists an element $t\in \mathcal{E}(U)$ such that
\[\rho^U_{U_\alpha}(t)\equiv t|_{U_\alpha}= t_\alpha,\]for every
$\alpha\in I$. It remains to show that $t$ is indeed an element of
$F(U)$. Suppose that $t$ is not an element of $F(U)$, so it follows
that $t$ cannot be written as $t= \sum_{k\in J\subseteq
I}^ma_{i_k}t^{i_k}$, with $J$ finite, $a_{i_k}\in \mathcal{A}(U)$,
and $t^{i_k}\in \mathcal{F}_{i_k}(U)$. This means that for some
$x\in U$, $t_x\equiv t(x)$ cannot be written as a linear combination
of finitely many $t_x^{i_k}\equiv t^{i_k}(x)$, where $t^{i_k}\in
\mathcal{F}_{i_k}(U)$ and $k\in J$ with $J$ a finite subset of $I$.
But this is a contradiction as $x\in U_\alpha$ for some $\alpha\in
I$, and $t|_{U_\alpha}= \sum_{k=1}^ma_{i_k}t^{i_k}$, where $m\in
\mathbb{N}$, $a_{i_k}\in \mathcal{A}(U_\alpha)$, and $t^{i_k}\in
\mathcal{F}_{i_k}(U_\alpha)$. Thus, $t\in F(U)$, and the proof is
finished.
\end{proof}

\begin{mydf}
\emph{Keeping with the notations of Lemma \ref{lem03}, we denote by
\[\mathcal{F}\equiv \sum_{i\in I}\mathcal{F}_i\]the
sub-$\mathcal{A}$-module, on $X$, of $\mathcal{E}$, generated by the
presheaf defined by $(\ref{eq12})$. We call the
\textit{sub-$\mathcal{A}$-module} $\sum_{i\in I}\mathcal{F}_i$ the
\textbf{sum of the family} $(\mathcal{F}_i)_{i\in I}$. In the case
where the index set $I$ is finite, say $I= \{1, \ldots, m\}$, we
shall often write $\sum_{i\in I}\mathcal{F}_i$ as
$\sum_{i=1}^m\mathcal{F}_i$, or $\mathcal{F}_1+\ldots +
\mathcal{F}_m$.} \hfill$\square$
\end{mydf}

On another side, it is readily verified that :
\begin{quote}
Given an $\mathcal{A}$-module $\mathcal{E}$ and a family
$(\mathcal{F}_i)_{i\in I}$ of sub-$\mathcal{A}$-modules, the
correspondence \begin{equation}U\longmapsto (\cap_{i\in
I}\mathcal{F}_i)(U)\equiv \cap_{i\in
I}(\mathcal{F}_i(U)),\label{eq13}\end{equation}where $U$ is any
open set in $X$, along with the obvious restriction maps yield a
\textit{complete presheaf of $\mathcal{A}$-modules} on $X$.
\end{quote}
The sheaf generated by the presheaf given by (\ref{eq13}) is
called the \textbf{intersection sub-$\mathcal{A}$-module} of the
\textit{family} $(\mathcal{F}_i)_{i\in I}$ and is denoted
\[\cap_{i\in I}\mathcal{F}_i.\]Thus, based on
Mallios[\cite{mallios}, Proposition 11.1, p. 51], one has
\[(\cap_{i\in I}\mathcal{F}_i)(U)= \cap_{i\in
I}(\mathcal{F}_i(U))\]for every open set $U\subseteq X$.

\begin{theo}\label{theo0.1}
Let $\mathcal{E}$ and $\mathcal{F}$ be free $\mathcal{A}$-modules
of finite rank on $X$, dual with respect to a non-degenerate
bilinear $\mathcal{A}$-form $\varphi$, and let $\mathcal{G}$ and
$\mathcal{H}$ be sub-$\mathcal{A}$-modules of $\mathcal{E}$. Then
\begin{enumerate}
\item[{$(a)$}] $\dim \mathcal{G}(U)+ \dim \mathcal{G}^\perp(U)=
\dim \mathcal{E}(U)$, for all open $U\subseteq X$. \item [{$(b)$}]
$(\mathcal{G}^\perp)^\perp= \mathcal{G}.$ \item [{$(c)$}]
$\mathcal{G}^\perp\cap \mathcal{H}^\perp= (\mathcal{G}+
\mathcal{H})^\perp$. \item [{$(d)$}] $(\mathcal{G}\cap
\mathcal{H})^\perp= \mathcal{G}^\perp+ \mathcal{H}^\perp.$ \item
[{$(e)$}]\footnote{Assertion $(e)$ is otherwise stated as
$\mathcal{E}$ and $\mathcal{F}$ having isomorphic
``\textit{projective geometries}", that is $p(\mathcal{E})$ is
isomorphic to $p(\mathcal{F})$. See Gruenberg and
Weir[\cite{weir}, p.29].} The mapping $\mathcal{G}\longmapsto
\mathcal{G}^\perp$ is a one-to-one mapping of the set of
sub-$\mathcal{A}$-modules of $\mathcal{E}$ onto the set of
sub-$\mathcal{A}$-modules of $\mathcal{F}$, such that
$\mathcal{G}\subseteq \mathcal{H}$ implies that
$\mathcal{G}^\perp\supseteq \mathcal{H}^\perp$. \item [{$(f)$}] If
$\mathcal{E}= \mathcal{G}\oplus \mathcal{H}$, then $\mathcal{F}=
\mathcal{G}^\perp\oplus \mathcal{H}^\perp.$ \item [{$(g)$}] The
$\mathcal{A}$-modules $\mathcal{G}$ and
$\mathcal{F}/\mathcal{G}^\perp$ are dual $\mathcal{A}$-modules
with respect to the non-degenerate bilinear $\mathcal{A}$-form
$\widetilde{\varphi}$, defined by
\[\widetilde{\varphi}^U(s, [t])= \varphi^U(s, t)\]
for all $s\in \mathcal{G}(U)$ and $\dot t\in
(\mathcal{F}/\mathcal{G}^\perp)(U)=
\mathcal{F}(U)/\mathcal{G}^\perp(U)$, with $U$ running over the
open subsets of $X$. \item [{$(h)$}] Suppose that $S\in
\mbox{End}_\mathcal{A}\mathcal{E}$ and $T\in
\mbox{End}_{\mathcal{A}}\mathcal{F}$ are transposes of each other
with respect to $\varphi$, and suppose that $\mathcal{G}$ is an
\textsf{$S$-invariant} sub-$\mathcal{A}$-module of $\mathcal{E}$,
i.e. \[S^U(s)\in \mathcal{G}(U)\]for all $s\in \mathcal{G}(U)$ and
$U$ open in $X$. Then, $\mathcal{G}^\perp$ is $T$-invariant, and
the restriction $S|_{\mathcal{G}}\equiv
\left(S^U|_{\mathcal{G}(U)}\right)_{X\supseteq U, open}$ and the
induced $\mathcal{A}$-morphism $T^\ast\equiv
T_{\mathcal{F}/\mathcal{G}^\perp}$ are transposes of each other
with respect to the bilinear $\mathcal{A}$-form
$\widetilde{\varphi}$, defined in part $(g)$.
\end{enumerate}
\end{theo}

\begin{proof}
$(a)$ Suppose that the rank of $\mathcal{E}$ is $n$ $(n\in
\mathbb{N})$, i.e. $\mathcal{E}= \mathcal{A}^n$ within an
$\mathcal{A}$-isomorphism, so that $\mathcal{E}(U)$ is
$\mathcal{A}(U)$-isomorphic to $\mathcal{A}^n(U)$ for every open
$U\subseteq X$. Now, let us fix an open subset $U$ of $X$; then we
have \[\mathcal{G}(U)=\mathcal{A}^{k(U)}(U)\equiv
\mathcal{A}^k(U)\]within an $\mathcal{A}(U)$-isomorphism and such
that $1\leq k\leq n$. Next, let $\{e_i^U\}_{1\leq i\leq n}$ be the
canonical basis of $\mathcal{E}(U)$, obtained from the Kronecker
gauge $\{\varepsilon^U_i\}_{1\leq i\leq n}$ (cf.
Mallios[\cite{mallios}, p.123]) through the
$\mathcal{A}$-isomorphism $\mathcal{E}= \mathcal{A}^n$. Since
$\mathcal{F}(U)$ is $\mathcal{A}(U)$-isomorphic to
$\mathcal{E}(U)= \mathcal{E}(U)^\ast= \mathcal{E}^\ast(U)$, we can
find, see Blyth[\cite{blyth}, Theorem 9.1, p.116], a basis
$\{f^U_j\}_{1\leq j\leq n}$ such that, using Proposition
\ref{prop01}, we have
\[\varphi^U(e^U_i, f^U_j)= \left\{\begin{array}{cc} 0^U, & i\neq j\\
1^U, & i=j.\end{array}\right.\]We assert that $\{f^U_{k+1},
\ldots, f^U_n\}$ is a basis of $\mathcal{G}(U)^\perp$. This is
clearly established as $\{e^U_i\}_{1\leq i\leq n}$ is a basis of
$\mathcal{G}(U)$, $\varphi^U(e^U_i, f^U_j)= 0^U$, for all $1\leq
i\leq k$, $k+1\leq j\leq n$, and $f_{k+1}^U, \ldots, f^U_n$ are
linearly independent and generate $\mathcal{G}(U)^\perp$. To see
that $f^U_{k+1}, \ldots, f_n^U$ generate $\mathcal{G}(U)^\perp$,
let $s= \sum_{i=1}^n\alpha_if_i^U\in \mathcal{G}(U)^\perp$. Since
$\varphi^U(e^U_i, s)= 0^U$, $1\leq i\leq k$, we have that
$\alpha_i= 0$ for $1\leq i\leq k$. Thus, assertion $(a)$ is
proved.

$(b)$ We have, for every open $U\subseteq X$,
\[\mathcal{G}(U)\subseteq
(\mathcal{G}^\perp(U))^\perp= (\mathcal{G}(U)^\perp)^\perp\equiv
\mathcal{G}(U)^{\perp\perp}.\]From part $(a)$, we have, for all
open $U\subseteq X$,
\begin{eqnarray*} \dim (\mathcal{G}^\perp(U))^\perp & =
& \dim \mathcal{F}(U)- \dim \mathcal{G}^\perp(U) \\ & = & \dim
\mathcal{F}(U)- (\dim \mathcal{E}(U)- \dim \mathcal{G}(U)) \\ & =
& \dim \mathcal{G}(U),
\end{eqnarray*}form which we deduce that for all open $U\subseteq
X$, \[\mathcal{G}(U)= (\mathcal{G}^\perp(U))^\perp\]within an
$\mathcal{A}(U)$-isomorphism. Hence, the
$\mathcal{A}(U)$-isomorphisms $\mathcal{G}(U)=
(\mathcal{G}^\perp(U))^\perp$, $U$ running over the open subsets
of $X$, along with the restriction maps $\sigma^U_V$ yield a
complete presheaf, defined by
\[U\longmapsto \mathcal{G}^{\perp\perp}\equiv
(\mathcal{G}(U))^{\perp\perp}:= (\mathcal{G}^\perp(U))^\perp.\]It
follows that if $(\mathcal{G}^\perp)^\perp\equiv
\mathcal{G}^{\perp\perp}$ is the sheaf corresponding to the
preceding (complete) presheaf, then we have
\[(\mathcal{G}^\perp)^\perp=
\mathcal{G}\]within an $\mathcal{A}$-isomorphism.

$(c)$ For every open $U\subseteq X$, one has \begin{eqnarray*}
(\mathcal{G}^\perp\cap \mathcal{H}^\perp)(U) & = &
\mathcal{G}^\perp(U)\cap \mathcal{H}^\perp(U) \\ & = &
\mathcal{G}(U)^\perp\cap \mathcal{H}(U)^\perp \\ & = &
(\mathcal{G}(U)+ \mathcal{H}(U))^\perp \\ & = &
((\mathcal{G}+\mathcal{H})(U))^\perp\\ & = & (\mathcal{G}+
\mathcal{H})^\perp(U);
\end{eqnarray*}it follows that $\mathcal{G}^\perp\cap
\mathcal{H}^\perp= (\mathcal{G}+ \mathcal{H})^\perp$ within an
$\mathcal{A}$-isomorphism.

$(d)$ is shown by combining $(b)$ and $(c)$.

$(e)$ Clearly for all open $U\subseteq X$,
$\mathcal{G}(U)\subseteq \mathcal{H}(U)$ implies that
\[\mathcal{G}^\perp(U)=
\mathcal{G}(U)^\perp \supseteq \mathcal{H}(U)^\perp=
\mathcal{H}^\perp(U).\] So, if
\[\{\left({(\sigma^\perp)}\right)^U_V:
\mathcal{G}^\perp(U)\longrightarrow \mathcal{G}^\perp(V)|\
\mbox{$V$, $U$ are open in $X$ and $V\subseteq U$}\}\] is the set
of restriction maps for the (complete) presheaf of sections of the
annihilator sheaf $\mathcal{G}^\perp$, then by taking
\[\left({\lambda^\perp}\right)^U_V:=
\left({\sigma^\perp}\right)^U_V|_{\mathcal{H}^\perp(U)}=
\rho^U_V|_{\mathcal{H}^\perp(U)}\]we obtain the (complete)
presheaf of sections of the sheaf $\mathcal{H}^\perp$. Therefore,
we have $\mathcal{H}^\perp\subseteq \mathcal{G}^\perp$. For the
one-to-one property, suppose that $\mathcal{G}^\perp =
\mathcal{H}^\perp$. Applying $(b)$, we have
\[\mathcal{G}= \left(\mathcal{G}^\perp\right)^\perp=
\left(\mathcal{H}^\perp\right)^\perp= \mathcal{H},\]where the
previous equalities are actually $\mathcal{A}$-isomorphisms. The
proof that every sub-$\mathcal{A}$-module $\mathcal{N}$ of the
$\mathcal{A}$-module $\mathcal{F}$ has the form
$\mathcal{G}^\perp$ for some sub-$\mathcal{A}$-module
$\mathcal{G}$ of $\mathcal{E}$ is immediate. In effect, applying
$(b)$, we have
\[\mathcal{N}=
\left(\mathcal{N}^\perp\right)^\perp\]within an
$\mathcal{A}$-isomorphism. Taking $\mathcal{G}= \mathcal{N}^\perp$
corroborates the assertion.

$(f)$ It suffices to show that if $U$ is an open subset of $X$,
then $\mathcal{E}(U)= \mathcal{G}(U)\oplus \mathcal{H}(U)$ implies
that $\mathcal{F}(U)= \mathcal{G}^\perp(U)\oplus
(\mathcal{H})^\perp(U)$. But this is shown in
Curtis[\cite{curtis}, p.242, part $(d)$ of proof of Theorem
(27.12)].

$(g)$ That $\widetilde{\varphi}$ is well defined is immediate. In
fact, fix an open subset $U$ of $X$; then for every $s\in
\mathcal{G}(U)$ and $\dot t= \dot{t'}\in
(\mathcal{F}/(\mathcal{E})^\perp)(U)=
\mathcal{F}(U)/(\mathcal{G})^\perp(U)$, we have
\[\varphi^U(s, t')= \varphi^U(s, t+z)= \varphi^U(s, t)\]since
$s\in \mathcal{G}(U)$, and $t'= t+z$ with $z\in
\mathcal{G}(U)^\perp= (\mathcal{G})^\perp(U)$. it is obvious that
$\widetilde{\varphi}$ is bilinear. The proof that
$\widetilde{\varphi}^U$ is non-degenerate for all open $U\subseteq
X$ can be found in Curtis[\cite{curtis}, p242, part $(e)$ of proof
of Theorem (27.12)].

$(h)$ We show first that $T(\mathcal{G}^\perp)\subseteq
\mathcal{G}^\perp$, that is
\[T(\mathcal{G}^\perp)(U):=
T^U(\mathcal{G}^\perp(U))= T^U(\mathcal{G}(U)^\perp)\subseteq
\mathcal{G}^\perp(U)= \mathcal{G}(U)^\perp,\]for all open
$U\subseteq X$. Let us consider arbitrarily any open subset
$U\subseteq X$, and let $s\in \mathcal{G}(U)$ and $t\in
\mathcal{G}(U)^\perp$. Then, \[\varphi^U(s, T^U(t))=
\varphi^U(S^U(s), t)= 0^U,\]because $S^U(\mathcal{G}(U))\subseteq
\mathcal{G}(U)$; therefore $T^U(\mathcal{G}(U)^\perp)\subseteq
\mathcal{G}(U)^\perp.$ For the remaining part of $(f)$, we start
by noticing that for every open $U\subseteq X$,
\[(T^\ast)^U\circ q^U= q^U\circ T^U\]where $q$ is the quotient
$A$-morphism $\mathcal{F}\longrightarrow
\mathcal{F}/\mathcal{G}^\perp$. It is sufficient to prove that for
$s\in \mathcal{G}(U)$, $t\in \mathcal{F}(U)$,
\[\widetilde{\varphi}^U(S^U|_{\mathcal{G}(U)}(s),
\dot t)= \widetilde{\varphi}^U(s, (T^\ast)^U(\dot t));\] this
statement is equivalent to showing that \[\varphi^U(S^U(s), t)=
\varphi^U(s, T^U(t)),\]which is exactly the condition that $S$ and
$T$ are transposes of each other.
\end{proof}

The last part of this section concerns with the
$\mathcal{A}$-isomorphism of the $\mathcal{A}$-annihilator of a
sub-$\mathcal{A}$-module $\mathcal{F}$ of an $\mathcal{A}$-module
$\mathcal{E}$ and the dual $(\mathcal{E}/\mathcal{F})^\ast$ of the
quotient $\mathcal{A}$-module $\mathcal{E}/\mathcal{F}$. This
question requires some preparation.

\begin{mydf}
\emph{Let $\mathcal{E}$ and $\mathcal{F}$ be $\mathcal{A}$-modules on a
topological space $X$, $U$ an open subset of $X$, and $\varphi\in
\mathcal{H}om_\mathcal{A}(\mathcal{E}, \mathcal{F})(U)=
\mbox{Hom}_{\mathcal{A}|_U}(\mathcal{E}|_U, \mathcal{F}|_U)$. For
any $\mathcal{A}$-module $\mathcal{G}$ on $X$, we define an
$\mathcal{A}(U)$-morphism
\[\xymatrix{\mathcal{H}om_\mathcal{A}(\mathcal{G}, \mathcal{E})(U)\ar[r]^{\varphi_\ast} &
\mathcal{H}om_\mathcal{A}(\mathcal{G}, \mathcal{F})(U)}\]by setting
\[\varphi_\ast(f)= \varphi\circ f\equiv (\varphi_V\circ
f_V)_{U\supseteq V, open}\equiv ((\varphi_V)_\ast
(f_V))_{U\supseteq V, open}
\]for all $f\in \mathcal{H}om_\mathcal{A}(\mathcal{G},
\mathcal{E})(U)$. Likewise, we can define an
$\mathcal{A}(U)$-morphism
\[\xymatrix{\mathcal{H}om_\mathcal{A}(\mathcal{E},
\mathcal{G})(U)\ar[r]^{\varphi^\ast} &
\mathcal{H}om_\mathcal{A}(\mathcal{F}, \mathcal{G})(U)}\]by the
assignment \[\varphi^\ast(f)= f\circ \varphi\equiv (f_V\circ
\varphi_V)_{U\supseteq V, open}\equiv
((\varphi_V)^\ast(f_V))_{U\supseteq V, open}\]for all $f\in
\mathcal{H}om_\mathcal{A}(\mathcal{E}, \mathcal{G})(U)$.}
\end{mydf}

\begin{prop}
Let $\mathcal{E}$, $\mathcal{F}$, and $\mathcal{G}$ be
$\mathcal{A}$-modules on $X$, $\varphi\in
\mathcal{H}om_\mathcal{A}(\mathcal{E}, \mathcal{F})(U)$ and
$\psi\in \mathcal{H}om_\mathcal{A}(\mathcal{F}, \mathcal{G})(U)$,
where $U$ is an open subset of $X$. Then, we have
\begin{enumerate} \item [{$(1)$}] $(\psi\circ \varphi)_\ast=
\psi_\ast\circ \varphi_\ast$ \item [{$(2)$}] $(\psi\circ
\varphi)^\ast= \varphi^\ast\circ \psi^\ast.$
\end{enumerate}
\end{prop}

\begin{proof}
Immediate.
\end{proof}

Our main interest in the above induced $\mathcal{A}(U)$-morphisms
$\varphi_\ast$ and $\varphi^\ast$ transpires in the following
theorem.

\begin{theo}
Consider a short exact sequence
\[\xymatrix{0\ar[r] & \mathcal{E}'\ar[r]^\varphi &
\mathcal{E}\ar[r]^\psi & \mathcal{E}''\ar[r] & 0}\]of
$\mathcal{A}$-modules $($on $X)$ and $\mathcal{A}$-morphisms. For
an arbitrary $\mathcal{A}$-module $\mathcal{F}$, the induced
sequences of $\mathcal{A}(X)$-modules and
$\mathcal{A}(X)$-morphisms
\begin{enumerate}
\item [{$(1)$}] $\xymatrix{0\ar[r] &
\mathcal{H}om_\mathcal{A}(\mathcal{F},
\mathcal{E}')(X)\ar[r]^{\varphi_\ast} &
\mathcal{H}om_\mathcal{A}(\mathcal{F},
\mathcal{E})(X)\ar[r]^{\psi_\ast} &
\mathcal{H}om_\mathcal{A}(\mathcal{F}, \mathcal{E}'')(X)}$
\item[{$(2)$}] $\xymatrix{0\ar[r] &
\mathcal{H}om_\mathcal{A}(\mathcal{E}'',
\mathcal{F})(X)\ar[r]^{\psi^\ast} &
\mathcal{H}om_\mathcal{A}(\mathcal{E},
\mathcal{F})(X)\ar[r]^{\varphi^\ast} &
\mathcal{H}om_\mathcal{A}(\mathcal{E}', \mathcal{F})(X)}$
\end{enumerate} are exact. The diagrams above are
$\mathcal{A}$-isomorphic to the diagrams
\begin{enumerate}
\item [{$(1')$}] $\xymatrix{0\ar[r] &
\mbox{Hom}_\mathcal{A}(\mathcal{F},
\mathcal{E}')\ar[r]^{\varphi_\ast} &
\mbox{Hom}_\mathcal{A}(\mathcal{F}, \mathcal{E})\ar[r]^{\psi_\ast}
& \mbox{Hom}_\mathcal{A}(\mathcal{F}, \mathcal{E}'')}$
\item[{$(2')$}] $\xymatrix{0\ar[r] &
\mbox{Hom}_\mathcal{A}(\mathcal{E}'',
\mathcal{F})\ar[r]^{\psi^\ast} &
\mbox{Hom}_\mathcal{A}(\mathcal{E},
\mathcal{F})\ar[r]^{\varphi^\ast} &
\mbox{Hom}_\mathcal{A}(\mathcal{E}', \mathcal{F})}$
\end{enumerate}
\label{30jtheo}\end{theo}

\begin{proof}
We shall show that $(1)$ is exact. The sequence $(2)$ is
established in a similar way.

First, let $f\in \ker\varphi_\ast$. We have $0= \varphi_\ast(f)=
\varphi\circ f\in \mathcal{H}om_\mathcal{A}(\mathcal{F},
\mathcal{E})(X)$, whence $f= (0_U)_{X\supseteq U, open}$ with
$0_U: \mathcal{F}(U)\longrightarrow \mathcal{E}(U)$, $0_U(s)= 0$
for all $s\in \mathcal{F}(U)$ and all open subset $U\subseteq X$,
which means that $\varphi_\ast$ is one-to-one.

Next, let us show that $\mbox{im}\varphi_\ast$ is an
$\mathcal{A}(X)$-submodule of the $\mathcal{A}(X)$-module
$\ker\psi_\ast$. (See Mallios[\cite{mallios}, pp 108, 109] for a
proof of the statement: Given $\mathcal{A}$-modules $\mathcal{E}$,
$\mathcal{F}$, and an $\mathcal{A}$-morphism $\phi:
\mathcal{E}\longrightarrow \mathcal{F}$. Then, $\ker\phi:= \{z\in
\mathcal{E}:\ \phi(z)=0\}$ and $\mbox{im}\phi:=
\phi(\mathcal{E})\subseteq \mathcal{F}$ are $\mathcal{A}$-modules;
consequently $\ker\phi(X)$ and $\mbox{im}\phi(X)$ are
$\mathcal{A}(X)$-modules.) If $f\in \mbox{im}\varphi_\ast$, there
exists $f'\in \mathcal{H}om_\mathcal{A}(\mathcal{F},
\mathcal{E}')(X)$ such that $f= \varphi_\ast(f')= \varphi\circ
f'$. Consequently
\[\psi_\ast(f)= \psi_\ast(\varphi_\ast(f'))= (\psi\circ
\varphi)_\ast(f')= 0\]because $\psi\circ \varphi= 0$. Thus, $f\in
\ker\psi_\ast$, and we have established that
\[\mbox{im}\varphi_\ast\subseteq \ker\psi_\ast.\]

Finally, let us show that $\ker\psi_\ast$ is an
$\mathcal{A}(X)$-submodule of $\mbox{im}\varphi_\ast$. To this end,
let $f\in \ker\psi_\ast$; then for every $s\in \mathcal{F}(U)$ where
$U$ is an open subset of $X$,
\[\psi_U(f_U(s))= (\psi_U\circ f_U)(s)= [(\psi_U)_\ast(f_U)](s)=
0\in \mathcal{E}''(U)\]and so $f_U(s)\in \ker\psi_U=
\mbox{im}\varphi_U$. Thus, there exists $s'\in \mathcal{E}'(U)$ such
that $f_U(s)= \varphi_U(s')$; and since $\varphi$ is one-to-one,
such an element $s'$ is unique. We can therefore define a mapping
$f'_U: \mathcal{F}(U)\longrightarrow \mathcal{E}'(U)$ by setting
$f'_U(s)= s'$. Clearly, $f_U'$ yields an $\mathcal{A}(U)$-morphism
of $\mathcal{A}(U)$-modules $\mathcal{F}(U)$ and $\mathcal{E}'(U)$,
which by abuse of language we also call $f'_U$. But \[f_U= \varphi_U
\circ f'_U= (\varphi_U)_\ast(f_U')\in
\mbox{im}(\varphi_U)_\ast,\]where $U$ is an arbitrary subset of $X$.
Thus, $\ker(\psi_U)_\ast\subseteq \mbox{im} (\varphi_U)_\ast$, for
every open $U\subseteq X$. Hence, $\ker\psi_\ast\subseteq
\mbox{im$\varphi_\ast$}$, which ends the proof.
\end{proof}

For the notion in the following definition, we refer to
Mallios[\cite{mallios}, pp. 301, 302] for specific details.

\begin{mydf}
\emph{Let $\mathcal{E}$ and $\mathcal{F}$ be $\mathcal{A}$-modules
on a topological space $X$. By the \textbf{transpose} of an
$\mathcal{A}$-morphism $\varphi: \mathcal{E}\longrightarrow
\mathcal{F}$, we mean the $\mathcal{A}$-morphism
\[{}^t\varphi\equiv ({}^t\varphi_U)_{X\supseteq U, open}:
\mathcal{F}^\ast\longrightarrow \mathcal{E}^\ast,\]given by the
assignment \[{}^t\varphi_U(u):= (u_V\circ \varphi_V)_{U\supseteq V,
open}\]for every $u\in \mathcal{F}^\ast(U)$, with $U$ open in
$X$.}\hfill$\square$
\end{mydf}

The principal properties of \textit{transposition} in classical
module theory apply in the setting as well, and are easily
verified.

\begin{prop}
Let $\mathcal{E}$, $\mathcal{F}$, $\mathcal{G}$ be
$\mathcal{A}$-modules on a topological space $X$. Then
\begin{enumerate}
\item [{$(1)$}] ${}^t(\mbox{id}_\mathcal{E})=
\mbox{id}_{\mathcal{E}^\ast}.$ \item [{$(2)$}] If $\varphi,
\psi\in \mbox{Hom}_\mathcal{A}(\mathcal{E}, \mathcal{F})$, then
${}^t(\varphi+ \psi)= {}^t\varphi+ {}^t\psi.$ \item[{$(3)$}] If
$\varphi\in \mbox{Hom}_\mathcal{A}(\mathcal{E}, \mathcal{F})$ and
$\psi\in \mbox{Hom}_\mathcal{A}(\mathcal{F}, \mathcal{G})$, then
${}^t(\psi\circ \varphi)= {}^t\varphi\circ {}^t\psi.$
\end{enumerate}
\label{classical}\end{prop}

\begin{cor}
If $\varphi: \mathcal{E}\longrightarrow \mathcal{F}$ is an
$\mathcal{A}$-isomorphism of the $\mathcal{A}$-modules
$\mathcal{E}$ and $\mathcal{F}$, then so is ${}^t\varphi:
\mathcal{F}^\ast\longrightarrow \mathcal{E}^\ast$; and we also
have in this case that $({}^t\varphi)^{-1}= {}^t(\varphi^{-1})$.
\end{cor}

\begin{proof}
By hypothesis and items $(1)$, $(3)$ of Proposition
\ref{classical}, we have \[{}^t\varphi\circ {}^t\varphi^{-1}=
{}^t(\varphi^{-1}\circ \varphi)= {}^t(\mbox{id}_\mathcal{E})=
\mbox{id}_{\mathcal{E}^\ast}\]and \[{}^t\varphi^{-1}\circ
{}^t\varphi= {}^t(\varphi\circ \varphi^{-1})=
{}^t(\mbox{id}_\mathcal{F})= \mbox{id}_{\mathcal{F}^\ast}.\]The
proof is finished.
\end{proof}

\begin{cor}\label{cor04}
If $\mathcal{F}$ is a sub-$\mathcal{A}$-module of $\mathcal{E}$,
then \[(\mathcal{E}/\mathcal{F})^\ast= \mathcal{F}^\perp,\]within
an $\mathcal{A}$-isomorphism.
\end{cor}

\begin{proof}
Let $q: \mathcal{E}\longrightarrow \mathcal{E}/\mathcal{F}$ be the
quotient $\mathcal{A}$-morphism and \[\xymatrix{0\ar[r] &
\mathcal{F} \ar[r]^\iota & \mathcal{E}\ar[r]^q &
\mathcal{E}/\mathcal{F}\ar[r] & 0}\]the natural short exact
sequence (see Mallios[\cite{mallios}, Lemma 2.1, p.116]). By
Theorem \ref{30jtheo}, we have the induced short exact sequence
\[\xymatrix{\mathcal{F}^\ast & \mathcal{E}^\ast\ar[l]_{{}^t\iota} &
(\mathcal{E}/\mathcal{F})^\ast\ar[l]_{{}^tq} & 0\ar[l]}.\]Then,
since
\[\ker ({}^t\iota)= (\mbox{im}\iota)^\perp,\]it follows from the
exactness of the foregoing sequence that
\[(\mathcal{E}/\mathcal{F})^\ast= \mbox{im}({}^tq)= \ker ({}^t\iota)=
(\mbox{im}\iota)^\perp= \mathcal{F}^\perp.\]
\end{proof}

\section{Properties of exterior $\mathcal{A}$-$2$-forms}

In this section, we examine some properties of exterior
$\mathcal{A}$-$2$-forms. The most useful property is the
\textit{normal} (or \textit{Darboux}) \textit{form} for exterior
$\mathcal{A}$-$2$-forms, see [\cite{malliosntumba}, Theorem 3.3],
which we prove this time by following Libermann and
Marle[\cite{libermann}, Theorem 2.3, pp. 4,5] and
Sternberg[\cite{sternberg}, Theorem 5.1, p. 24].

Throughout this section, $\mathcal{E}$ stands for an
\textit{$\mathcal{A}$-module} on a topological space $X\equiv (X,
\tau)$, and $\omega:\mathcal{E}\oplus \mathcal{E}\longrightarrow
\mathcal{A}$ for an \textit{exterior bilinear $\mathcal{A}$-form},
unless otherwise specified.

\begin{mydf}
\emph{Let $\eta: \mathcal{E}\oplus \ldots\oplus
\mathcal{E}\longrightarrow \mathcal{A}$ be a non-zero exterior
$\mathcal{A}$-$k$-form. For every $s\equiv (s_U)_{U\in
\mathcal{T}}\in \prod_{U\in \mathcal{T}}\mathcal{E}(U)$, let
\[i: Hom_\mathcal{A}(\mathcal{E},
Hom_\mathcal{A}(\bigwedge{}^k\mathcal{E}^\ast,
\bigwedge{}^{(k-1)}\mathcal{E}^\ast))\]be an $\mathcal{A}$-morphism,
whose $U$-component, for an arbitrary open subset $U$ of $X$, is the
$\mathcal{A}(U)$-morphism \begin{eqnarray*}\lefteqn{i_U\in
Hom_{\mathcal{A}(U)}(\mathcal{E}(U),
Hom_{\mathcal{A}(U)}((\bigwedge{}^k\mathcal{E}^\ast)(U),
(\bigwedge{}^{(k-1)}\mathcal{E}^\ast)(U)))}\\ & & \equiv
Hom_{\mathcal{A}(U)}(\mathcal{E}(U),
Hom_{\mathcal{A}(U)}((\bigwedge{}^k\mathcal{E})^\ast(U),
(\bigwedge{}^{(k-1)}\mathcal{E}^)\ast(U))),\end{eqnarray*}which is
given by
\[i(s_U)\eta_U({s_1}_U,\ldots, {s_{k-1}}_U)= \eta_U(s_U,{s_1}_U,\ldots,
{s_{k-1}}_U)\]for all ${s_1}_U,\ldots, {s_{k-1}}_U\in
\mathcal{E}(U)$. We call \[i(s)\eta\equiv
(i_U(s_U)\eta_U)_{U\in\tau}: \mathcal{E}\oplus\ldots \oplus
\mathcal{E}\longrightarrow \mathcal{A}\] the \textbf{inner
$\mathcal{A}$-product} of $\eta$ and $s$.}\hfill$\square$
\end{mydf}

Now, let us move our attention to $\mathcal{A}$-$2$-forms. Suppose
$\omega: \mathcal{E}\oplus \mathcal{E}\longrightarrow \mathcal{A}$
is an $\mathcal{A}$-$2$-form on a free $\mathcal{A}$-module
$\mathcal{E}$; for a family $s\equiv (s_U)_{U\in \tau}\in
\prod_{U\in \tau}\mathcal{E}(U)$, the following mapping in
$\mbox{Hom}_\mathcal{A}(\mathcal{E},
\mathcal{E}^{\ast\ast}=\mathcal{E}= \mathcal{E}^\ast)$, see
Mallios[\cite{mallios}, relation 5.4, p. 298], given by
\[s\longmapsto -i(s)\omega\equiv -(i(s_U)\omega_U)_{U\in
\tau},\] will be denoted, keeping with Libermann and
Marle[\cite{libermann}, p. 3], by $\omega^\flat$. Next, for every
open $U\subseteq X$, we consider the canonical basis
$(e_i^U)_{1\leq i\leq n}\subseteq \mathcal{E}(U)$. Suppose that
$\omega\equiv (\omega^U)_{U\in \mathcal{T}}$ is such that
\[\omega_{ij}^U\equiv \omega^U(e_i^U, e_j^U)\]for any open subset
$U\subseteq X$, and such that \[ \mbox{rank}\ (\omega^U_{ij})=
\mbox{rank}\ (\omega^V_{ij}),\]then the \textbf{rank} of $\omega$ is
by definition the rank of the matrix $(\omega_{ij}^U)$ for any open
$U\subseteq X$. Throughout this paper, \textit{all exterior
$\mathcal{A}$-$2$-forms }$\omega: \mathcal{E}\oplus
\mathcal{E}\longrightarrow \mathcal{A}$ on the free
$\mathcal{A}$-module $\mathcal{E}$ of rank $n$ \textit{are assumed
to have a rank}. We shall call such $\mathcal{A}$-$2$-forms
\textbf{rankwise} $\mathcal{A}$-$2$-forms.

As in Libermann and Marle[\cite{libermann}, pp. 3, 4], given an
exterior $\mathcal{A}$-$2$-form $\omega:\mathcal{E}\oplus
\mathcal{E}\longrightarrow \mathcal{A}$ on a free
$\mathcal{A}$-module $\mathcal{E}$, we denote by
${}^\flat\mathcal{E}$ the sub-$\mathcal{A}$-module im
$\omega^\flat\subseteq \mathcal{E}^\ast= \mathcal{E}$; see
Mallios[\cite{mallios}, p. 109] for a proof of the following
statement: \begin{quote} \textit{If $\varphi\equiv (\varphi_U)\in
Hom_\mathcal{A}(\mathcal{E}, \mathcal{F})$, then im $\varphi:=
\varphi(\mathcal{E})$ is a subsheaf of the sheaf
$\mathcal{F}$.}\end{quote}

Since $\omega^\flat_U(\mathcal{E}(U))$ is an
$\mathcal{A}(U)$-submodule of $\mathcal{E}^\ast(U)=
\mathcal{E}(U)^\ast= \mathcal{E}(U)$ for every open $U\subseteq
X$, it follows that im $\omega^\flat$ is a
sub-$\mathcal{A}$-module of $\mathcal{E}^\ast= \mathcal{E}$.

It is worth noting too that ${}^\flat\mathcal{E}=
(\ker\omega^\flat)^\perp$ within an $\mathcal{A}$-isomorphism. If
$\ker\omega^\flat\neq 0$, we have \[\mathcal{E}/\ker\omega^\flat=
{}^\flat\mathcal{E}\]within an $\mathcal{A}$-isomorphism, see
Mallios[\cite{mallios}, Lemma 2.1, p. 116, relation (2.19), p. 110].

The first part of the following theorem was proved in our previous
paper, see \cite{malliosntumba}, however, here, we are presenting
another proof for the same first part of the theorem; the
relevance of this approach consists in the fact it provides hints,
which are necessary for the proof of the second part of the
theorem. This theorem in its classical form is proved in Libermann
and Marle[\cite{libermann}, Theorem 2.3, p. 4] and
Sternberg[\cite{sternberg}, Theorem 5.1, p. 24].

\begin{theo}
Let $(X, \mathcal{A}, \mathcal{P}, |\cdot |)$ be an ordered
$\mathbb{R}$-algebraized space, endowed with an \textsf{absolute
value morphism} $($see $\cite{malliosntumba}$ $)$, such that every
strictly positive section of $\mathcal{A}$ is invertible.
Moreover, let $\omega$ be a rankwise $\mathcal{A}$-$2$-form on the
free $\mathcal{A}$-module $\mathcal{E}$ of rank $n$. Then, for
every $x\in X$, there exist an open neighborhood $U\subseteq X$ of
$x$ and a basis
\[\begin{array}{ll}s^1_U, \ldots, s^{2m}_U\in
{}^\flat\mathcal{E}(U), & 2\leq 2m \leq n\end{array}\] such that
\[\omega_U= \sum_{k=1}^ms^{2k-1}_U\wedge
s^{2k}_U;\]furthermore, $s^2_U$ may be chosen arbitrarily in
${}^\flat\mathcal{E}(U)$.
\end{theo}

\begin{proof}
Let $({e_1}_X,\ldots, {e_n}_X)\equiv (e_1,\ldots, e_n)$ be a basis
 of $\mathcal{E}(X)$, whose corresponding dual basis is $(e^1_X,\ldots,
e^n_X)\equiv (e^1, \ldots, e^n)$. The $X$-component of the
$\mathcal{A}$-$2$-form $\omega$ may be expressed as
\[\omega_X= \frac{1}{2}\sum_{(i, j)}a_{ij}e^i\wedge e^j,\]where the
coefficients $a_{ij}$ are sections of $\mathcal{A}$ over $X$, i.e.
$a_{ij}\in \mathcal{A}(X)\equiv \Gamma(X, \mathcal{A})$, and satisfy
the condition $a_{ji}= -a_{ij}$, $1\leq i, j\leq n$. By hypothesis,
at every $x\in X$, the coefficients $a_{ij}(x)$ are not all zero.
Let us fix a point $x\in X$; we can rearrange the basis $(e^1,
\ldots, e^n)$ as to obtain $a_{12}(x)\neq 0$. Therefore, for some
open neighborhood $U$ of $x$, we have \[|\rho^X_U(a_{12})|\in
\mathcal{P}^\ast(U),\]where $\mathcal{P}^\ast:=
\mathcal{P}-\{0\}\subseteq \mathcal{A}^\bullet$, cf.
Mallios[\cite{mallios}, relation (10.1), p. 335], and the
$(\rho^V_W)_{V\supseteq W, open}$, with $V$ running over the open
subsets of $X$, are the restriction maps for the $($complete$)$
presheaf of sections of the coefficient sheaf $\mathcal{A}$. Let us
assume likewise that the restriction maps for the $($complete$)$
presheaf of sections of $\mathcal{E}$ are maps
$(\sigma^V_W)_{V\supseteq W, open}$, with $V$ being any open set in
$X$; we set
\[s^1_U= \frac{1}{|\rho^X_U(a_{12})|}i(\sigma^X_U(e_1))\omega_U,\]and
\[s^2_U= i(\sigma^X_U(e_2))\omega_U\]i.e. \[s^1_U= \sigma^X_U(e^2)+
\frac{1}{|\rho^X_U(a_{12})|}\sum_{k=3}^n\rho^X_U(a_{1k})\sigma^X_U(e^k),\]and
\[s^2_U= -|\rho^X_U(a_{12})|\sigma^X_U(e^1)+
\sum_{k=3}^n\rho^X_U(a_{2k})\sigma^X_U(e^k).\] It is clear that
$(s^1_U, s^2_U, \sigma^X_U(e^3), \ldots, \sigma^X_U(e^n))$ is a
basis of $\mathcal{E}(U)$. Next, we set \[{\omega_1}_U= \omega_U-
s^1_U\wedge s^2_U;\]${\omega_1}_U$ does not contain any expression
involving $\sigma^X_U(e^1)$ or $\sigma^X_U(e^2)$. If
${\omega_1}_U=0$, then $\omega_U= s^1_U\wedge s^2_U$, and we are
done. Otherwise, we continue the same process until we achieve the
desired form, that is if \[{\omega_1}_U=
\frac{1}{2}\sum_{i,j=3}^nb_{ij}\rho^X_U(e^i)\wedge
\rho^X_U(e^j_U)\]with $b_{ji}=- b_{ij}\in \mathcal{A}(U)$, $3\leq i,
j\leq n$, then there exists a $b_{ij}\in \mathcal{A}(U)$ such that
$b_{ij}(x)\neq 0$. As above, there exists an open neighborhood
$V\subseteq U$ such that $\rho^U_V(b_{ij})\neq 0$. Through a
convenient rearrangement of the basis vectors $\rho^X_U(e^3),
\ldots, \rho^X_U(e^n)$, we may assume that $\rho^X_U(b_{34})\neq 0$.
So, as before, we shall get an $\mathcal{A}$-$2$-form
\[{\omega_2}_V= {\omega_1}_V- s^3_V\wedge s^4_V,\]where
\[\begin{array}{ll} s^3_V=
\frac{1}{|\rho^U_V(b_{34})|}i(\sigma^U_V(e_3){\omega_1}_V, & s^4_V=
i(\sigma^U_V(e_4){\omega_1}_V,\end{array}\]and so on $\cdots$

Let $t$ be a non-zero element in ${}^\flat\mathcal{E}(U)$. There
exists a non-zero vector $s_2\in \mathcal{E}(U)$ such that $t=
i(s_2)\omega_U$. Since $t\neq 0$, there exists a section $s_1\in
\mathcal{E}(U)$ such that $t(s_1)\neq 0$, hence $\omega_U(s_1,
s_2)\neq 0$. We choose the basis $(e_1, \ldots, e_n)$ of
$\mathcal{E}(X)$ such that $\sigma^X_U(e_1)= s_1$, and
$\sigma^X_U(e_2)= s_2$ so that $s^2_U= t$.
\end{proof}

\section{Symplectic Reduction}

We start by observing that Definition \ref{def05} hints that if
$\mathcal{E}$ is a free $\mathcal{A}$-module of rank $n$, then an
$\mathcal{A}$-bilinear morphism $\omega: \mathcal{E}\oplus
\mathcal{E}\longrightarrow \mathcal{A}$ is non-degenerate if and
only if $\omega$ is rankwise, and of rank $n$. A pair
$(\mathcal{E}, \omega)$, where $\mathcal{E}$ is an arbitrary
$\mathcal{A}$-module and $\omega: \mathcal{E}\oplus
\mathcal{E}\longrightarrow \mathcal{A}$ a non-degenerate
$\mathcal{A}$-bilinear morphism, is called a \textbf{symplectic
$\mathcal{A}$-module}.

Throughout this section, we will be particularly interested in
symplectic free $\mathcal{A}$-modules of finite rank.

The most important examples of sub-$\mathcal{A}$-modules of a
symplectic $\mathcal{A}$-module $(\mathcal{E}, \omega)$
($\mathcal{E}$ is not assumed necessarily free) are the following

\begin{mydf}
\emph{Let $(\mathcal{E}, \omega)$ be a symplectic $\mathcal{A}$-module and
$\mathcal{F}\subseteq \mathcal{E}$ a sub-$\mathcal{A}$-module. We
say that
\begin{enumerate}
\item [{$(i)$}] $\mathcal{F}$ is \textbf{isotropic} if
$\mathcal{F}\subseteq \mathcal{F}^\perp$, that is
$\omega|_{\mathcal{F}\oplus \mathcal{F}}\equiv
\omega|_\mathcal{F}= 0$. \item [{$(ii)$}] $\mathcal{F}$ is
\textbf{co-isotropic} if $\mathcal{F}^\perp\subseteq \mathcal{F}$,
that is $\omega|_{\mathcal{F}^\perp}= 0$. \item[{$(iii)$}]
$\mathcal{F}$ is a \textbf{symplectic sub-$\mathcal{A}$-module} if
$\omega|_\mathcal{F}: \mathcal{F}\oplus \mathcal{F}\longrightarrow
\mathcal{A}$ is non-degenerate. \item [{$(iv)$}] $\mathcal{F}$ is
\textbf{Lagrangian} if it is isotropic and has an isotropic
complement, that is $\mathcal{E}= \mathcal{F}\oplus \mathcal{G}$,
where $\mathcal{G}$ is isotropic.\hfill$\square$
\end{enumerate}}\label{def2.1}
\end{mydf}

The next result will often be used to define \textit{Lagrangian
sub-$\mathcal{A}$-modules.}

\begin{prop}
Let $(\mathcal{E}, \omega)$ be a symplectic $\mathcal{A}$-module of
finite rank on $X$, and $\mathcal{F}\subseteq \mathcal{E}$ a
sub-$\mathcal{A}$-module. Then, the following assertions are
equivalent: \begin{enumerate}\item [{$(i)$}] $\mathcal{F}$ is
Lagrangian. \item [{$(ii)$}] $\mathcal{F}= \mathcal{F}^\perp$,
within an $\mathcal{A}$-isomorphism. \item [{$(iii)$}] $\mathcal{F}$
is isotropic and rank $\mathcal{F}= \frac{1}{2}$ rank $\mathcal{E}$.
\end{enumerate}
\end{prop}

\begin{proof}
The following proof is derived from the proof of Proposition 5.3.3,
in Abraham and Marsden[\cite{abraham}, p. 404].

First, we prove that $(i)$ implies $(ii)$. We have
$\mathcal{F}\subseteq \mathcal{F}^\perp$ by hypothesis. Next, we
have to show the converse, i.e. $\mathcal{F}^\perp\subseteq
\mathcal{F}$. To this end, for every open subset $U$ of $X$, let
$s^U\in \mathcal{F}^\perp(U)= \mathcal{F}(U)^\perp$; since
$\mathcal{E}(U)= \mathcal{F}(U)+ \mathcal{G}(U)$, where
$\mathcal{G}$, according to Definition \ref{def2.1}$(iv)$, is an
isotropic complement of $\mathcal{F}$, write $s^U= s_0^U+ s_1^U$
for some $s_0^U\in \mathcal{F}(U)$ and $s_1^U\in \mathcal{G}(U)$.
We shall show that $s_1^U=0$. Indeed, let $s_1^U\in
\mathcal{G}(U)\subseteq \mathcal{G}^\perp(U)=
\mathcal{G}(U)^\perp$, and $s_1^U= s^U- s_0^U\in
\mathcal{F}^\perp(U)$. Thus,
\begin{eqnarray*} s_1^U \in \mathcal{G}^\perp(U)\cap
\mathcal{F}^\perp(U) & = & (\mathcal{G}(U)+ \mathcal{F}(U))^\perp,
\mbox{by virtue of Theorem \ref{theo0.1}$(c)$}\\ & = &
\mathcal{E}^\perp(U) \\ & = & \{0\}, \mbox{by the non-degeneracy of
$\omega_U$}.\end{eqnarray*} Thus, $s_1^U= 0$, so
$\mathcal{F}^\perp(U)\subseteq \mathcal{F}(U)$; since $U$ is
arbitrary, $\mathcal{F}^\perp\subseteq \mathcal{F}$, that is $(ii)$
holds.

The implication $(ii)\Longrightarrow (iii)$ is immediate.

Finally, we prove that $(iii)$ implies $(i)$. First, observe that
$(iii)$ implies $\dim \mathcal{F}(U)= \dim \mathcal{F}^\perp(U)$ for
any open subset $U$ of $X$. Since $\mathcal{F}\subseteq
\mathcal{F}^\perp$, we have that $\mathcal{F}= \mathcal{F}^\perp$.
Now, we construct the isotropic complement $\mathcal{G}$ of
$\mathcal{F}$ as follows. For every open $U\subseteq X$, choose
arbitrarily $s_1^U\notin \mathcal{F}(U)$, and let
\[\mathcal{F}_1(U):= \{as_1^U|\ a\in \mathcal{A}(U)\}\equiv
\mathcal{A}s_1^U.\]It is easy to see that the correspondence
\begin{equation}\label{eq14}U\longmapsto
\mathcal{F}_1(U)\end{equation} along with the obvious restrictions
yield a complete presheaf of $\mathcal{A}$-modules. (If $\rho^U_V:
\mathcal{E}(U)\longrightarrow \mathcal{E}(V)$ is a restriction map,
$\rho^U_V|_{\mathcal{F}_1(U)}: \mathcal{F}_1(U)\longrightarrow
\mathcal{F}_1(V)$ is the corresponding restriction map for the
presheaf defined in (\ref{eq14}).) The sheaf $\mathcal{F}_1$
generated by the presheaf defined in (\ref{eq14}) is clearly a free
$\mathcal{A}$-module of rank $1$. For every open $U\subseteq X$,
$\mathcal{F}(U)\cap \mathcal{F}_1(U)= \{0\}$; consequently
\begin{eqnarray*}\mathcal{F}(U)^{\perp\perp}+ \mathcal{F}(U)^\perp &
= & (\mathcal{F}(U)^\perp\cap \mathcal{F}_1(U))^\perp, \mbox{by
virtue of Theorem \ref{theo0.1}}\\ & = & (\mathcal{F}(U)\cap
\mathcal{F}_1(U))^\perp, \mbox{since $\mathcal{F}=
\mathcal{F}^\perp$} \\ & = & \{0\}^\perp \\ & = &
\mathcal{E}(U).\end{eqnarray*}Now, choose, for every open
$U\subseteq X$, an element $s_2^U\in \mathcal{F}_1(U)^\perp=
\mathcal{F}_1^\perp(U)$ such that $s_2^U\notin \mathcal{F}(U)+
\mathcal{F}_1(U)$. Next, let \[\mathcal{F}_2(U):= \mathcal{F}_1(U)+
\mathcal{A}s_2^U;\]proceed inductively as before until one gets
\begin{equation}\label{eq15} \mathcal{F}(U)+ \mathcal{F}_k(U)=
\mathcal{E}(U)\end{equation}for every open $U\subseteq X$. The
correspondence \[U\longmapsto \mathcal{F}_2(U)\]defines a complete
presheaf of $\mathcal{A}$-modules. The sheaf $\mathcal{F}_2$
generated by the foregoing presheaf is a free $\mathcal{A}$-module
of rank $2$. Similarly, the sheaf $\mathcal{F}_k$ obtained by
sheafifying the complete presheaf, given by \[U\longmapsto
\mathcal{F}_k(U)\]is a free $\mathcal{A}$-module of rank $k$.
Equation (\ref{eq15}) yields the following $\mathcal{A}$-isomorphism
\[\mathcal{F}+ \mathcal{F}_k= \mathcal{E}.\]By construction,
$\mathcal{F}(U)\cap \mathcal{F}_k(U)= \{0\}$ for every open
$U\subseteq X$, so $\mathcal{E}= \mathcal{F}\oplus \mathcal{F}_k$.
Also, by construction, \begin{eqnarray*} \mathcal{F}_2^\perp(U)=
\mathcal{F}_2(U)^\perp & = & (\mathcal{F}_1(U)+
\mathcal{A}s_2^U)^\perp \\ & = & \mathcal{F}_1(U)^\perp+
(\mathcal{A}s_2^U)^\perp \\ & \supseteq & \mathcal{A}s_1^U+
\mathcal{A}s_2^U \\ & = & \mathcal{F}_2(U).\end{eqnarray*} It
follows that $\mathcal{F}_2\subseteq \mathcal{F}_2^\perp$. In the
same way, one shows that $\mathcal{F}_k$ is isotropic as well.
Thus, $\mathcal{E}= \mathcal{F}\oplus \mathcal{F}_k$, with
$\mathcal{F}_k\subseteq \mathcal{F}_k^\perp$ as desired.
\end{proof}

\begin{lem}
Let $(\mathcal{E}, \omega)$ be a symplectic $\mathcal{A}$-module,
and $\mathcal{F}\subseteq \mathcal{E}$ a sub-$\mathcal{A}$-module.
Then, $\mathcal{F}/\mathcal{F}\cap \mathcal{F}^\perp$ has a natural
symplectic structure.
\end{lem}

\begin{proof}
Indeed, let \begin{equation}\label{eq16}\widehat{\omega}_U(s+
(\mathcal{F}\cap \mathcal{F}^\perp)(U), s'+ (\mathcal{F}\cap
\mathcal{F}^\perp)(U)):= \omega_U(s, s')\end{equation}for all $s,
s'\in \mathcal{F}(U)$, where $U$ is an open set in $X$. Equation
(\ref{eq16}) can equivalently be written as \[\widehat{\omega}_U(s+
(\mathcal{F}(U)\cap \mathcal{F}^\perp(U)), s'+ (\mathcal{F}(U)\cap
\mathcal{F}^\perp(U)))= \omega_U(s, s'),\]because $\mathcal{F}\cap
\mathcal{F}^\perp$ is a sub-$\mathcal{A}$-module of $\mathcal{E}$
and $(\mathcal{F}\cap \mathcal{F}^\perp)(U)= \mathcal{F}(U)\cap
\mathcal{F}^\perp(U)= \mathcal{F}(U)\cap \mathcal{F}(U)^\perp$ for
all open subset $U\subseteq X$. Now, denote by
\[((\mathcal{F}/\mathcal{F}\cap \mathcal{F}^\perp)(U):=
\mathcal{F}(U)/\mathcal{F}(U)\cap \mathcal{F}^\perp(U),
\sigma^U_V)\]the (complete) presheaf of sections associated with the
sheaf $\mathcal{F}/\mathcal{F}\cap\mathcal{F}^\perp$, and by
\[(\mathcal{A}(U), \lambda^U_V)\]the corresponding presheaf of
sections for the coefficient sheaf $\mathcal{A}$. It is clearly easy
to see that \[\xymatrix{(\mathcal{F}/\mathcal{F}\cap
\mathcal{F}^\perp)(U)\oplus (\mathcal{F}/\mathcal{F}\cap
\mathcal{F}^\perp)(U)\ar[r]^{\hspace{27mm}\widehat{\omega}_U}\ar[d]_{\sigma^U_V\oplus
\sigma^U_V} & \mathcal{A}(U)\ar[d]^{\lambda^U_V}\\
(\mathcal{F}/\mathcal{F}\cap \mathcal{F}^\perp)(V)\oplus
(\mathcal{F}/\mathcal{F}\cap
\mathcal{F}^\perp)(V)\ar[r]_{\hspace{27mm}\widehat{\omega}_V} &
\mathcal{A}(V)}\]commutes for all open subsets $U, V\subseteq X$
such that $V\subseteq U$. Thus, \[\widehat{\omega}:
\mathcal{F}/\mathcal{F}\cap \mathcal{F}^\perp\oplus
\mathcal{F}/\mathcal{F}\cap\mathcal{F}^\perp\longrightarrow
\mathcal{A}\] is an $\mathcal{A}$-morphism.

We need now show that $\widehat{\omega}$ is well defined and is a
symplectic $\mathcal{A}$-form. Indeed, let $t, t'\in
\mathcal{F}(U)\cap \mathcal{F}^\perp(U)$, where $U$ is open in $X$;
then \begin{eqnarray*} \omega_U(s+t, s'+ t') & = & \omega_U(s, s')+
\omega_U(s+t, t')+ \omega_U(t, s') \\ & = & \omega_U(s,
s'),\end{eqnarray*}since $\omega_U(s+t, t')=0= \omega_U(t, s')$.
Thus, $\widehat{\omega}$ is well defined. It is easy to see that
$\widehat{\omega}$ is $\mathcal{A}$-bilinear. Let us now show that
$\widehat{\omega}$ is non-degenerate. Suppose $s\in \mathcal{F}(U)$
such that \begin{equation}\widehat{\omega}_U(s+
\mathcal{F}(U)\mathcal{F}^\perp(U), s'+ \mathcal{F}(U)\cap
\mathcal{F}^\perp(U))= 0\label{eq17}\end{equation} for all $s'\in
\mathcal{F}(U)$. By virtue of the definition of $\widehat{\omega}$,
see (\ref{eq16}), Equation (\ref{eq17}) becomes
\[\omega_U(s, s')=0\] for all $s'\in \mathcal{F}(U)$. Therefore,
$s\in \mathcal{F}^\perp(U)$, so in
$(\mathcal{F}/\mathcal{F}\cap\mathcal{F}^\perp)(U)$, is zero.
\end{proof}

We now introduce some terminology in connection with the preceding
lemma.

\begin{mydf}
\emph{Let $(\mathcal{E}, \omega)$ be a symplectic
$\mathcal{A}$-module, and $\mathcal{F}\subseteq \mathcal{E}$ a
\textit{co-isotropic sub-$\mathcal{A}$-module} of $\mathcal{E}$.
The symplectic $\mathcal{A}$-module
$(\mathcal{F}/\mathcal{F}^\perp, \widehat{\omega})$, where
$\widehat{\omega}$ is given by (\ref{eq16}), is called a
\textbf{reduced symplectic $\mathcal{A}$-module} or the
\textbf{$\mathcal{A}$-module $\mathcal{E}$ reduced by
$\mathcal{F}$}. The notation $\mathcal{E}/\mathcal{F}$ will also
be used to denote the underlying $\mathcal{A}$-module
$\mathcal{F}/\mathcal{F}^\perp$ of the reduced symplectic
$\mathcal{A}$-module $(\mathcal{F}/\mathcal{F}^\perp,
\widehat{\omega})$.}\hfill$\square$
\end{mydf}

\begin{prop}
Let $(\mathcal{E}, \omega)$ be a symplectic free
$\mathcal{A}$-module of finite rank, $\mathcal{G}\subseteq
\mathcal{E}$ a Lagrangian sub-$\mathcal{A}$-module and
$\mathcal{F}\subseteq \mathcal{E}$ a co-isotropic
sub-$\mathcal{A}$-module of $\mathcal{E}$. Then,
\[(\mathcal{G}\cap \mathcal{F})/\mathcal{F}^\perp\subseteq
\mathcal{E}_\mathcal{F}\]is Lagrangian in the reduced symplectic
$\mathcal{A}$-module.
\end{prop}

\begin{proof}
Take an open subset $U\subseteq X$ and $s, s'\in (\mathcal{G}\cap
\mathcal{F})(U)= \mathcal{G}(U)\cap \mathcal{F}(U)$. One has
\[\widehat{\omega}_U(s+ \mathcal{F}^\perp(U), s'+
\mathcal{F}^\perp(U))= \omega_U(s, s')=0;\]therefore
$(\mathcal{G}\cap \mathcal{F})/\mathcal{F}^\perp$ is isotropic.

Next, we need show that \[\dim ((\mathcal{G}\cap
\mathcal{F})/\mathcal{F}^\perp)(U)= \frac{1}{2}\dim
(\mathcal{F}/\mathcal{F}^\perp)(U),\]for every open $U\subseteq
X$, to complete the proof of the proposition. The proof of this
fact can be found in Abraham and Marsden[\cite{abraham},
Proposition 5.3.10, pp 407-408].
\end{proof}

\noindent Anastasios Mallios\\ {Department of Mathematics} \\
{University of Athens}\\ {Athens, Greece}\\
{Email: amallios@math.uoa.gr}

 \noindent Patrice P.
Ntumba\\{Department of Mathematics and Applied
Mathematics}\\{University of Pretoria}\\ {Hatfield 0002, Republic of
South Africa}\\{Email: patrice.ntumba@up.ac.za}


\addcontentsline{toc}{section}{REFERENCES}
\begin{thebibliography}{99}
\bibitem{abraham} R. Abraham, J.E. Marsden: \textit{Foundations of
Mechanics}. The Benjamin/ Cummings Publishing Company, Inc, London,
1978.
\bibitem{adkins} W.A. Adkins, S.H. Weintraub: \textit{Algebra. An approach via Module
Theory.}Springer-Verlag New York, Inc. 1992.
\bibitem{blyth} T.S. Blyth: \textit{Module Theory. An approach to linear algebra. Second
edition}. Oxford Science Publications. Clarendon Press. Oxford 1990.
\bibitem{cartan} E. Cartan: \textit{Le\c{o}ns sur les invariants
int\'{e}graux.} Hermann, Paris, 1968 (premi\`{e}re \'{e}dition:
1922).
\bibitem{curtis} CW. Curtis: \textit{Linear Algebra. An introductory approach. Third
edition.}Allyn and Bacon, Inc. Boston, 1974.
\bibitem{weir} K.W. Gruenberg, A.K. Weir: \textit{Linear Geometry. Second
Edition.}Springer, 1977.
\bibitem{lang} S. Lang: \textit{Algebra. Revised Third
Edition}. Springer, 2002.
\bibitem{libermann} P. Libermann, C-M. Marle: \textit{Symplectic Geometry and Analytical
Mechanics} D. Reidel Publishing Company, Dordrecht, 1987.
\bibitem{maclane} S. Mac Lane: \textit{Categories for the Working Mathematician. Second
Edition} Springer Verlag New York, Inc. 1998.
\bibitem{mallios} A. Mallios: \textit{Geometry of Vector Sheaves}. An Axiomatic Approach to Differential Geometry.
Volume $I$: Vector Sheaves. General Theory. Kluwer Academic
Publishers. Netherlands, 1998.
\bibitem{localizing} A. Mallios: \textit{On localizing topological
algebras.} Topological algebras and their applications, 79- 95,
Contemp. Math., \textbf{341}, Amer. Math. Soc., Providence, RI,
2004.
\bibitem{malliosntumba} A. Mallios, P.P. Ntumba:
\textit{Fundamentals for Symplectic $\mathcal{A}$-modules. Affine
Darboux Theorem} Under refereeing.
\bibitem{sternberg} S. Sternberg: \textit{Lectures on Differential
Geometry} AMS Chelsea Publishing, Providence,Rhode Island, 1999.

\end{thebibliography}
\end{document}